\documentclass[10pt]{amsart}
\pdfoutput=1
\usepackage{amssymb,amscd,graphicx,enumerate,epsfig,color}
\usepackage[bookmarks=true]{hyperref}

\newcommand{\Claim}[1]{\noindent\textbf{Claim. }\textit{#1}\\}

\newtheorem{theorem}{Theorem}[section]

\newtheorem{lemma}[theorem]{Lemma}

\theoremstyle{definition}
\newtheorem{definition}[theorem]{Definition}

\newcommand{\D}{\mathcal{D}}
\newcommand{\DV}{\mathcal{D}_\mathcal{V}}
\newcommand{\DW}{\mathcal{D}_\mathcal{W}}
\newcommand{\DVW}{\mathcal{D}_{\mathcal{VW}}}
\newenvironment{proofc}{\noindent\textit{Proof of Claim.}}{\\}

\newcommand{\Case}[1]{\textbf{Case #1.}}

\newcommand{\V}{\mathcal{V}}
\newcommand{\W}{\mathcal{W}}

\begin{document}

\title[]{On the disk complexes of weakly reducible, unstabilized Heegaard splittings of genus three I - the Structure Theorem}

\author{Jungsoo Kim}
\date{February 21, 2015}

\begin{abstract}
	Let $(\V,\W;F)$ be a weakly reducible, unstabilized, genus three Heegaard splitting in an orientable, irreducible $3$-manifold $M$ and $\DVW(F)$ the subset of the disk complex $\D(F)$ consisting of simplices having at least one vertex from $\V$ and at least one vertex from $\W$.
	In this article, we describe the shape of $\DVW(F)$ and prove that there is a function from the components of $\DVW(F)$ to the isotopy classes of the generalized Heegaard splittings obtained by weak reductions from $(\V,\W;F)$, where this function describes how the thick levels are embedded in the relevant compression bodies.
\end{abstract}

\address{\parbox{4in}{
	BK21 PLUS SNU Mathematical Sciences Division,\\ Seoul National University\\ 
	1 Gwanak-ro, Gwanak-Gu, Seoul 151-747, Korea\\
}} 
	
\email{pibonazi@gmail.com}

\maketitle
\section{Introduction and Result}
Throughout this paper, all surfaces and $3$-manifolds will be taken to be compact and orientable.
In \cite{8}, Hempel introduced the ``\textit{distance}'' of a Heegaard splitting $(\V,\W;F$), where it comes from the simplicial complex of isotopy classes of the compressing disks of $F$, say the ``\textit{disk complex}'' $\D(F)$.
The distance has been mainly used for \textit{strongly irreducible} Heegaard splittings after the Hempel's work where it is the condition that every compressing disk of $\V$ must intersect every compressing disk of $\W$, because the distance is just either $0$ or $1$ in the cases of \textit{weakly reducible} (i.e. not strongly irreducible) Heegaard splittings.
But Bachman made use of the disk complex to analyze  weakly reducible Heegaard splittings by introducing the concept a ``\textit{critical surface}'' in \cite{1} \cite{2}, and he gave the proof of \textit{Gordon's Conjecture} in \cite{2}.
Moreover, he generalized this idea to the concept a ``\textit{topologically minimal surface}'' in \cite{3} and the ``\textit{topological index}'' of a topologically minimal surface $F$ refers to the homotopy index of $\D(F)$.
Especially, this idea gave counterexamples of \textit{the Stabilization Conjecture} by using the concept a ``\textit{$g$-barrier surface}'' in \cite{7}.
In \cite{4} \cite{5} \cite{6} , he proved that there is a resemblance between a topologically minimal surface and a geometrically minimal surface and therefore these results give a strong connection between differential geometry and topology in a $3$-manifold.
Hence, the importance of analyzing the exact shape of the components of the subset of $\D(F)$ consisting of simplices having at least one vertex from $\V$ and at least one vertex from $\W$, say $\DVW(F)$, arose after the Bachman's works.
But there has been no known result on  what information does each component of $\DVW(F)$ contain other than the topological index.
In this article, we will find another information contained in  the components of $\DVW(F)$.

Let $M$ be an orientable, irreducible $3$-manifold and $(\V,\W;F)$ a weakly reducible, unstabilized, genus three Heegaard splitting in $M$.
In \cite{11}, the author proved that (i) $\D(F)$ is contractible if there is only one generalized Heegaard splitting obtained by weak reduction from $(\V,\W;F)$  and the embedding of each thick level in the relevant compression body is also unique up to isotopy or (ii) $\pi_1(\D(F))$ is non-trivial otherwise.
Indeed, we can find specific descriptions of $\DVW(F)$ for all cases when $\D(F)$ is contractible in \cite{11}.
In this article, motivated by this idea, we will describe the exact shape of each component of $\DVW(F)$ in general case by using the concept a ``\textit{building block}'' of $\DVW(F)$ (see Definition \ref{definition-3-3}, Definition \ref{definition-3-5} and Definition \ref{definition-3-6}) where each building block corresponds to only one generalized Heegaard splitting obtained by weak reduction from $(\V,\W;F)$ up to isotopy in Theorem \ref{theorem-3-13}.
Moreover, each component of $\DVW(F)$ is contractible.

\begin{theorem}[Theorem \ref{theorem-3-13}]\label{theorem-1-1}
Let $(\V,\W;F)$ be a weakly reducible, unstabilized, genus three Heegaard splitting in an orientable, irreducible $3$-manifold $M$.
Then we can characterize the components of $\DVW(F)$ into five types, where each component is contractible.
Moreover, there is a uniquely determined weak reducing pair representing each component of $\DVW(F)$.
\end{theorem}

We can refer to Figure \ref{figure1} for the shapes of the components of $\DVW(F)$.
\begin{figure}
\includegraphics[width=12cm]{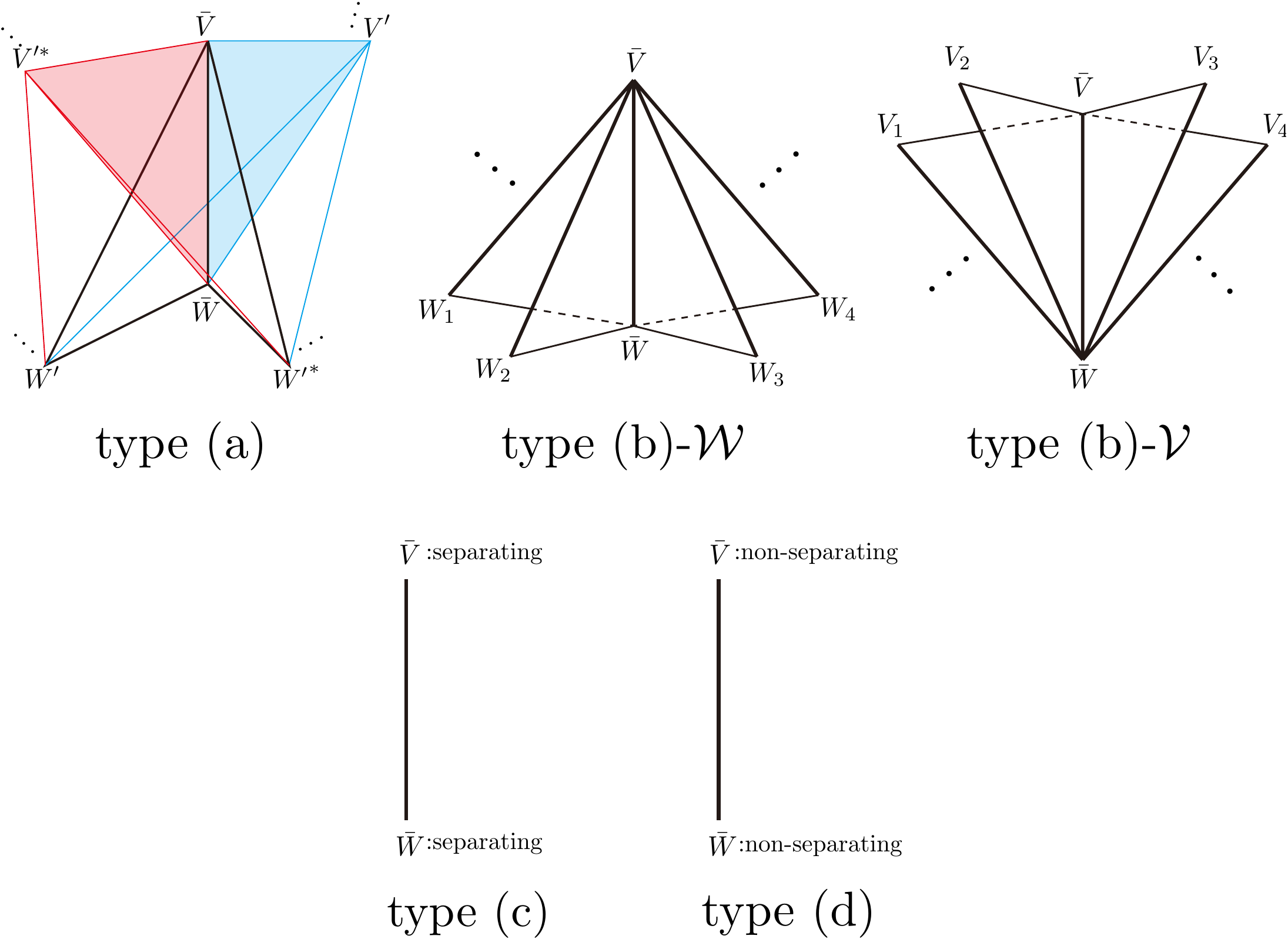}
\caption{the five types of components of $\DVW(F)$\label{figure1}}
\end{figure}

As a result, we prove the following theorem which reveals a hidden structure behind $\DVW(F)$.

\begin{theorem}[Theorem \ref{theorem-3-16}, the Structure Theorem]\label{theorem-1-2}
Let $(\V,\W;F)$ be a weakly reducible, unstabilized, genus three Heegaard splitting in an orientable, irreducible $3$-manifold $M$.
Then there is a function from the components of $\DVW(F)$ to the isotopy classes of the generalized Heegaard splittings obtained by weak reductions from $(\V,\W;F)$.
The number of components of the preimage of an isotopy class of this function is the number of ways to embed the thick level contained in $\V$ into $\V$ (or in $\W$ into $\W$).
This means that if we consider a generalized Heegaard splitting $\mathbf{H}$ obtained by weak reduction from $(\V,\W;F)$, then the way to embed the thick level of $\mathbf{H}$ contained in $\V$ into $\V$ determines the way to embed the thick level of $\mathbf{H}$ contained in $\W$ into $\W$ up to isotopy and vise versa.
\end{theorem}

\section{Preliminaries\label{section2}}

\begin{definition}
A \textit{compression body} (\textit{generalized compression body} resp.) is a $3$-manifold which can be obtained by starting with some closed, orientable, connected surface $F$, forming the product $F\times I$, attaching some number of $2$-handles to $F\times\{1\}$ and capping off all (\textit{some} resp.) resulting $2$-sphere boundary components that are not contained in $F\times\{0\}$ with $3$-balls. 
The boundary component $F\times\{0\}$ is referred to as $\partial_+$. 
The rest of the boundary is referred to as $\partial_-$. 
\end{definition}

\begin{definition}
A \textit{Heegaard splitting} of a $3$-manifold $M$ is an expression of $M$ as a union $\V\cup_F \W$, denoted   as $(\V,\W;F)$,  where $\V$ and $\W$ are compression bodies that intersect in a transversally oriented surface $F=\partial_+\V=\partial_+\W$. 
We say $F$ is the \textit{Heegaard surface} of this splitting. 
If $\V$ or $\W$ is homeomorphic to a product, then we say the splitting  is \textit{trivial}. 
If there are compressing disks $V\subset \V$ and $W\subset \W$ such that $V\cap W=\emptyset$, then we say the splitting is \textit{weakly reducible} and call the pair $(V,W)$ a \textit{weak reducing pair}. 
If $(V,W)$ is a weak reducing pair and $\partial V$ is isotopic to $\partial W$ in $F$, then we call $(V,W)$ a \textit{reducing pair}.
If the splitting is not trivial and we cannot take a weak reducing pair, then we call the splitting \textit{strongly irreducible}. 
If there is a pair of compressing disks $(\bar{V},\bar{W})$ such that $\bar{V}$ intersects $\bar{W}$ transversely in a point in $F$, then we call this pair a \textit{canceling pair} and say the splitting is \textit{stabilized}. 
Otherwise, we say the splitting is \textit{unstabilized}.
\end{definition}

\begin{definition}
Let $F$ be a surface of genus at least two in a compact, orientable $3$-manifold $M$. 
Then the \emph{disk complex} $\D(F)$ is defined as follows: 
\begin{enumerate}[(i)]
\item Vertices of $\D(F)$ are isotopy classes of compressing disks for $F$.
\item A set of $m+1$ vertices forms an $m$-simplex if there are representatives for each
that are pairwise disjoint.
\end{enumerate}
\end{definition}

\begin{definition}
Consider a Heegaard splitting $(\V,\W;F)$ of an orientable, irreducible $3$-manifold $M$. 
Let $\DV(F)$ and $\DW(F)$ be the subcomplexes of $\D(F)$ spanned by compressing disks in $\V$ and $\W$ respectively. 
We call these subcomplexes \textit{the disk complexes of $\V$ and $\W$}.
Let $\DVW(F)$ be the subset  of $\D(F)$ consisting of the simplices having at least one vertex from $\DV(F)$ and at least one vertex from $\DW(F)$.
\end{definition}

\begin{theorem}[McCullough, \cite{13}]
$\DV(F)$ and $\DW(F)$ are contractible.
\end{theorem}

Note that $\D(F)=\DV(F)\cup \DVW(F)\cup \DW(F)$.
Hence, $\D(F)$ may have a non-trivial homotopy type in general (see Figure \ref{figure0}).
\begin{figure}
\includegraphics[width=12cm]{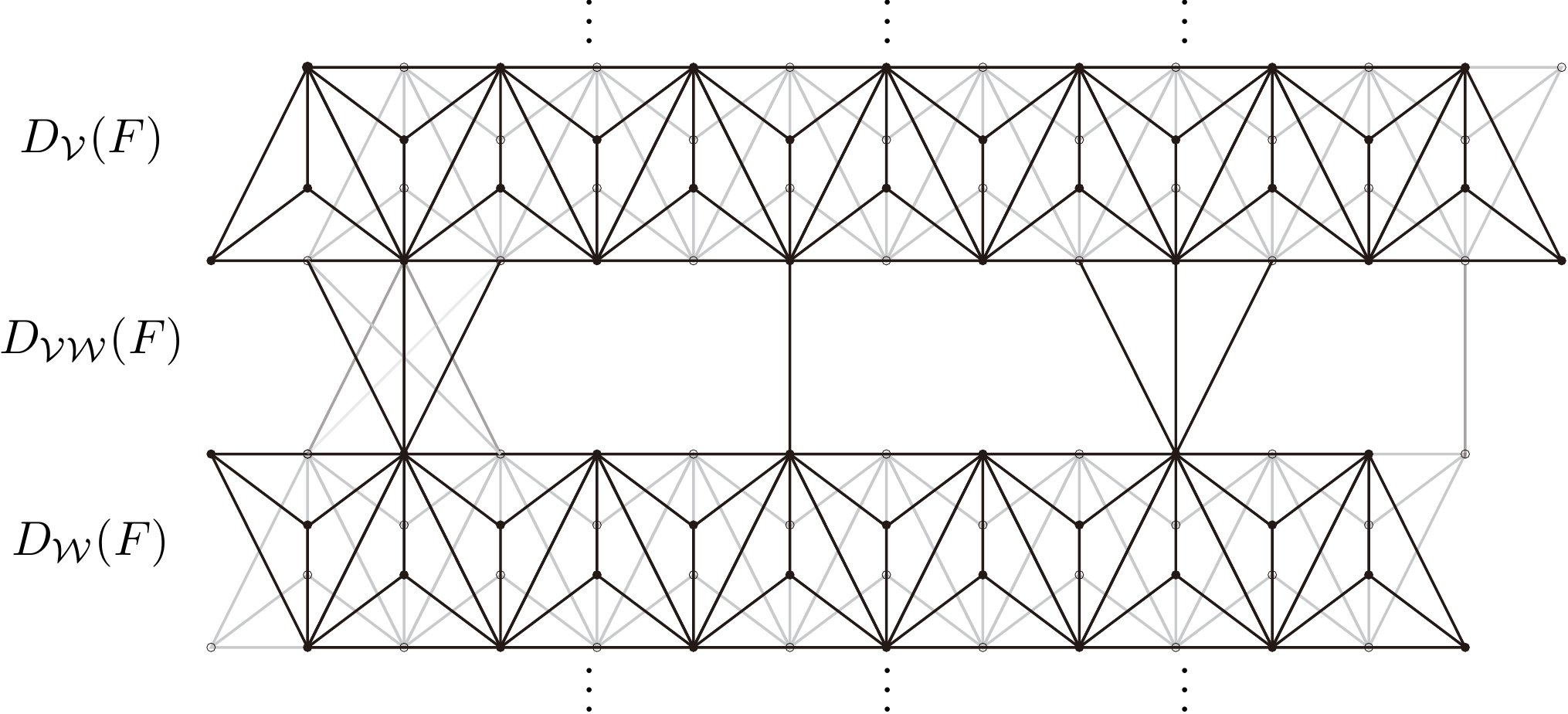}
\caption{$\D(F)$ may have a non-trivial homotopy type in general.\label{figure0}}
\end{figure}

From now on, we will consider only unstabilized Heegaard splittings of an irreducible $3$-manifold. 
If a Heegaard splitting of a compact $3$-manifold is reducible, then the manifold is reducible or the splitting is stabilized (see \cite{14}).
Hence, we can exclude the possibilities of reducing pairs among weak reducing pairs.

\begin{definition}
Suppose $W$ is a compressing disk for $F\subset M$. 
Then there is a subset of $M$ that can be identified with $W\times I$ so that $W=W\times\{\frac{1}2\}$ and $F\cap(W\times I)=(\partial W)\times I$. 
We form the surface $F_W$, obtained by \textit{compressing $F$ along $W$}, by removing $(\partial W)\times I$ from $F$ and replacing it with $W\times(\partial I)$. 
We say the two disks $W\times(\partial I)$ in $F_W$ are the $\textit{scars}$ of $W$. 
\end{definition}

\begin{lemma}[Lustig and Moriah, Lemma 1.1 of \cite{12}] \label{lemma-2-7}
Suppose that $M$ is an irreducible $3$-manifold and $(\V,\W;F)$ is an unstabilized Heegaard splitting of $M$. 
If $F'$ is obtained by compressing $F$ along a collection of pairwise disjoint disks, then no $S^2$ component of $F'$ can have scars from disks in both $\V$ and $\W$. 
\end{lemma}

\begin{lemma}[J. Kim, Lemma 2.9 of \cite{9}]\label{lemma-2-8}
Suppose that $M$ is an irreducible $3$-manifold and $(\V,\W;F)$ is an unstabilized, genus three  Heegaard splitting of $M$. 
If there exist three mutually disjoint compressing disks $V$, $V'\subset\V$ and $W\subset \W$, then either $V$ is isotopic to $V'$, or one of $\partial V$ and $\partial V'$ bounds a punctured torus $T$ in $F$ and  the other is a non-separating loop in $T$. 
Moreover, we cannot choose three weak reducing pairs $(V_0, W)$, $(V_1,W)$, and $(V_2,W)$ such that $V_i$ and $V_j$ are mutually disjoint and non-isotopic in $\V$ for $i\neq j$. 
\end{lemma}

The next is the definition of a ``\textit{generalized Heegaard splitting}'' originated from \cite{15}.

\begin{definition}[Definition 4.1 of \cite{2}]\label{definition-2-9}
A \textit{generalized Heegaard splitting} (GHS) $\mathbf{H}$ of a $3$-manifold $M$ is a pair of sets of pairwise disjoint, transversally oriented, connected surfaces, $\operatorname{Thick}(\mathbf{H})$ and $\operatorname{Thin}(\mathbf{H})$ (called the \textit{thick levels} and \textit{thin levels}, resp.), which satisfies the following conditions.
\begin{enumerate}
\item Each component $M'$ of $M-\operatorname{Thin}(\mathbf{H})$ meets a unique element $H_+$ of $\operatorname{Thick}(\mathbf{H})$ and $H_+$ is a Heegaard surface in $M'$.
Henceforth we will denote the closure of the component of $M-\operatorname{Thin}(\mathbf{H})$ that contains an element $H_+\in\operatorname{Thick}(\mathbf{H})$ as $M(H_+)$.
\item As each Heegaard surface $H_+\subset M(H_+)$ is transversally oriented, we can consistently talk about the points of $M(H_+)$ that are ``above''  $H_+$ or ``below'' $H_+$.
Suppose $H_-\in\operatorname{Thin}(\mathbf{H})$.
Let $M(H_+)$ and $M(H_+')$ be the submanifolds on each side of $H_-$.
Then $H_-$ is below $H_+$ if and only if it is above $H_+'$.
\item There is a partial ordering on the elements of $\operatorname{Thin}(\mathbf{H})$ which satisfies the following: Suppose $H_+$ is an element of $\operatorname{Thick}(\mathbf{H})$, $H_-$ is a component of $\partial M(H_+)$ above $H_+$ and $H_-'$ is a component of $\partial M(H_+)$ below $H_+$.
Then $H_->H_-'$.
\end{enumerate}
We call $\operatorname{Thin}(\mathbf{H})-\{\partial M\}$  the \textit{inner thin level}.
\end{definition}

\begin{definition}[Definition 5.1 of \cite{2}]
Let $M$ be a compact, orientable $3$-manifold.
Let $\mathbf{G}=\{T(\mathbf{G}),t(\mathbf{G})\}$ be a pair of sets of transversally oriented, connected surfaces in $M$ such that the elements of $T(\mathbf{G})\cup t(\mathbf{G})$ are pairwise disjoint.
Then we say $\mathbf{G}$ is a \textit{pseudo-GHS} if the following hold.
\begin{enumerate}
\item Each component $M'$ of $M-t(\mathbf{G})$ meets exactly one element $G_+$ of $T(\mathbf{G})$.
We denote the closure of $M'$ as $M(G_+)$.
\item Each element $G_+\in T(\mathbf{G})$ separates $M(G_+)$ into generalized (possibly trivial) compression bodies $\W$ and $\W'$, where $\partial_+\W=\partial_+\W'=G_+$.
\item There is a partial ordering of the elements of $t(\mathbf{G})$ that satisfies similar properties to the partial ordering of the thin levels of a GHS given in Definition \ref{definition-2-9}.
\end{enumerate}
\end{definition}

\begin{definition}[a restricted version of Definition 5.2, Definition 5.3 and Definition 5.6 of \cite{2}]\label{definition-2-11}
Let $M$ be a compact, orientable 3-manifold.
Let $\mathbf{H}$ be a Heegaard splitting of $M$, i.e. $\operatorname{Thick}(\mathbf{H})=\{F\}$ and $\operatorname{Thin}(\mathbf{H})$ consists of $\partial M$.
Let $V$ and $W$ be disjoint compressing disks of $F$ from the opposite sides of $F$ such that ${F}_{VW}$ has no $2$-sphere component. 
(Lemma \ref{lemma-2-7} guarantees that ${F}_{VW}$ will not have a $2$-sphere component in the proof of Theorem \ref{theorem-1-1}.)
Define
$$T(\mathbf{G'})=(\operatorname{Thick}(\mathbf{H})-\{F\})\cup\{{F}_{V}, {F}_{W}\}, \text{ and}$$
$$t(\mathbf{G'})=\operatorname{Thin}(\mathbf{H})\cup\{{F}_{VW}\},$$
where we assume that each element of $T(\mathbf{G'})$ belongs to the interior of $\V$ or $\W$ by slightly pushing off $F_V$ or $F_W$ into the interior of $\V$ or $\W$ respectively and then also assume that they miss $F_{VW}$.
We say the pseudo-GHS $\mathbf{\mathbf{G'}}=\{T(\mathbf{G'}),t(\mathbf{G'})\}$ is obtained from $\mathbf{H}$ by \textit{pre-weak reduction} along $(V,W)$.
Hence, the relative position of the elements of $T(\mathbf{G'})$ and $t(\mathbf{G'})$ follows the order described in Figure \ref{figure2}.
We can imagine that the compressing disk $V$ ($W$ resp.) of $F$ would become a compressing disk of ${F}_{W}$ (${F}_{V}$ resp.) in the solid between ${F}_W$ (${F}_V$ resp.) and ${F}_{VW}$  after the pre-weak reduction by slightly extending $V$ ($W$ resp.) as the dashed line in the right of Figure \ref{figure2}.
\begin{figure}
\includegraphics[width=10cm]{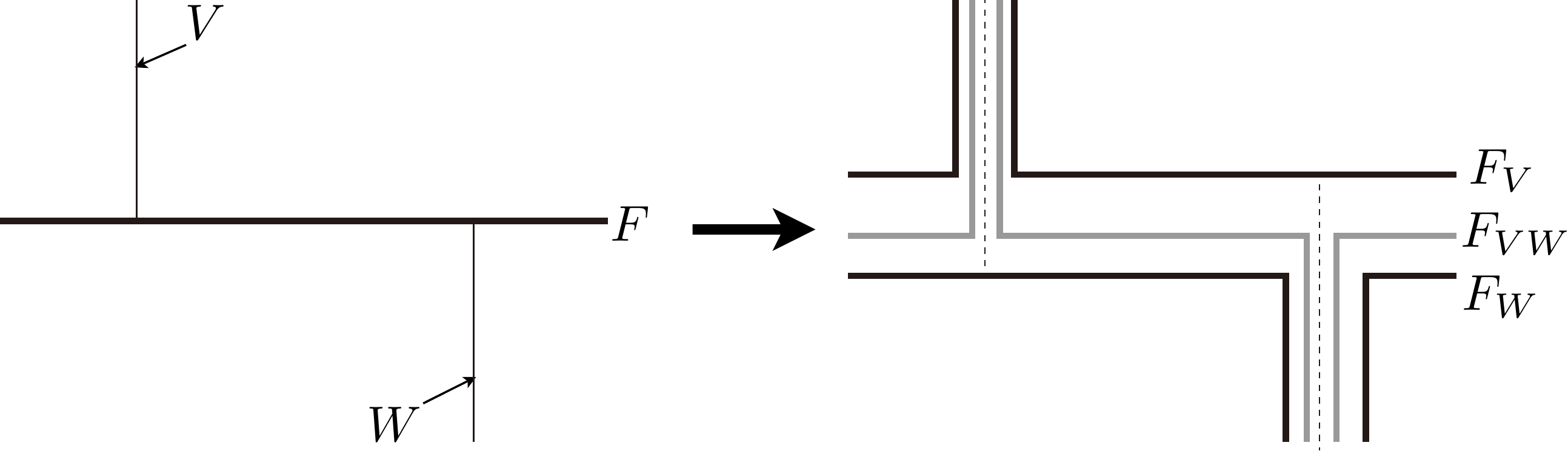}
\caption{pre-weak reduction \label{figure2}}
\end{figure}
If there are elements $S\in T(\mathbf{G'})$ and $s\in t(\mathbf{G'})$ that cobound a product region $P$ of $M$ such that $P\cap T(\mathbf{G'})=S$ and $P\cap t(\mathbf{G'})=s$, then remove $S$ from $T(\mathbf{G'})$ and $s$ from $t(\mathbf{G'})$.
This gives a GHS $\mathbf{G}$ of $M$ from the pseudo-GHS $\mathbf{G'}$ (see Lemma 5.4 of \cite{2}) and we say $\mathbf{G}$ is obtained from $\mathbf{G'}$ by \textit{cleaning}.
We say the GHS $\mathbf{G}$ of $M$ given by pre-weak reduction along $(V,W)$, followed by cleaning, is obtained from $\mathbf{H}$ by \textit{weak reduction} along $(V,W)$.
\end{definition}
We can refer to Figure \ref{figure6} as an example of a weak reduction from a genus three, weakly reducible Heegaard splitting $(\V,\W;F)$ along a weak reducing pair $(V,W)$ when $V$ cuts off $(\text{torus})\times I$ from $\V$ and $W$ cuts off a solid torus from $\W$. 

\begin{definition}[J. Kim, Definition 2.12 of \cite{10}]
In a weak reducing pair for a Heegaard splitting $(\V,\W;F)$, if a disk belongs to $\V$, then we call it a \emph{$\V$-disk}.
Otherwise, we call it a \emph{$\W$-disk}.	
We call a $2$-simplex in $\DVW(F)$ represented by two vertices in $\DV(F)$ and one vertex in $\DW(F)$ a \textit{$\V$-face}, and also define a \textit{$\W$-face} symmetrically.
Let us consider a $1$-dimensional graph as follows.
\begin{enumerate}
\item We assign a vertex to each $\V$-face in $\DVW(F)$.
\item If a $\V$-face shares a weak reducing pair with another $\V$-face, then we assign an edge between these two vertices in the graph.
\end{enumerate}
We call this graph the \emph{graph of $\V$-faces}.
If there is a maximal subset $\varepsilon_\V$ of $\V$-faces in $\DVW(F)$ representing a connected component of the graph of $\V$-faces and the component is not an isolated vertex, then we call $\varepsilon_\V$ a \emph{$\V$-facial cluster}.
Similarly, we define the \emph{graph of $\W$-faces} and  a \textit{$\W$-facial cluster}.
In a $\V$-facial cluster, every weak reducing pair gives the common $\W$-disk, and vise versa.
\end{definition}

If we consider an unstabilized, genus three Heegaard splitting of an irreducible $3$-manifold, then we get the following lemmas.

\begin{lemma}[J. Kim, Lemma 2.13 of \cite{10}]\label{lemma-2-13}
Suppose that $M$ is an irreducible $3$-manifold and $(\V,\W;F)$ is an unstabilized, genus three Heegaard splitting of $M$.
If there are two $\V$-faces $f_1$ represented by $\{V_0,V_1,W\}$ and $f_2$ represented by $\{V_1, V_2, W\}$ sharing a weak reducing pair $(V_1,W)$, then $\partial V_1$ is non-separating, and $\partial V_0$, $\partial V_2$ are separating in $F$.
Therefore, there is a unique weak reducing pair in a $\V$-facial cluster which can belong to two or more faces in the $\V$-facial cluster.
\end{lemma}

\begin{definition}[J. Kim, Definition 2.14 of \cite{10}]\label{definition-2-14}
By Lemma \ref{lemma-2-13}, there is a unique weak reducing pair in a $\V$-facial cluster belonging to two or more faces in the $\V$-facial cluster.
We call it the \textit{center} of a $\V$-facial cluster.
We call the other weak reducing pairs \textit{hands} of a $\V$-facial cluster.
See Figure \ref{figure3}.
Note that if a $\V$-face in a $\V$-facial cluster is represented by two weak reducing pairs, then one is the center and the other is a hand.
Lemma \ref{lemma-2-13} means that the $\V$-disk in the center of a $\V$-facial cluster is non-separating, and those from hands are all separating.
Moreover, Lemma \ref{lemma-2-8} implies that (i) the $\V$-disk in a hand of a $\V$-facial cluster is a band-sum of two parallel copies of that of the center of the $\V$-facial cluster and (ii) the $\V$-disk of a hand of a $\V$-facial cluster determines that of the center of the $\V$-facial cluster by the uniqueness of the meridian disk of the solid torus which the $\V$-disk of the hand cuts off from $\V$.
\end{definition}
	
Note that every $\V$ - or $\W$- facial cluster is contractible in $\D(F)$ (see Figure \ref{figure3}).

\begin{figure}
\includegraphics[width=4.5cm]{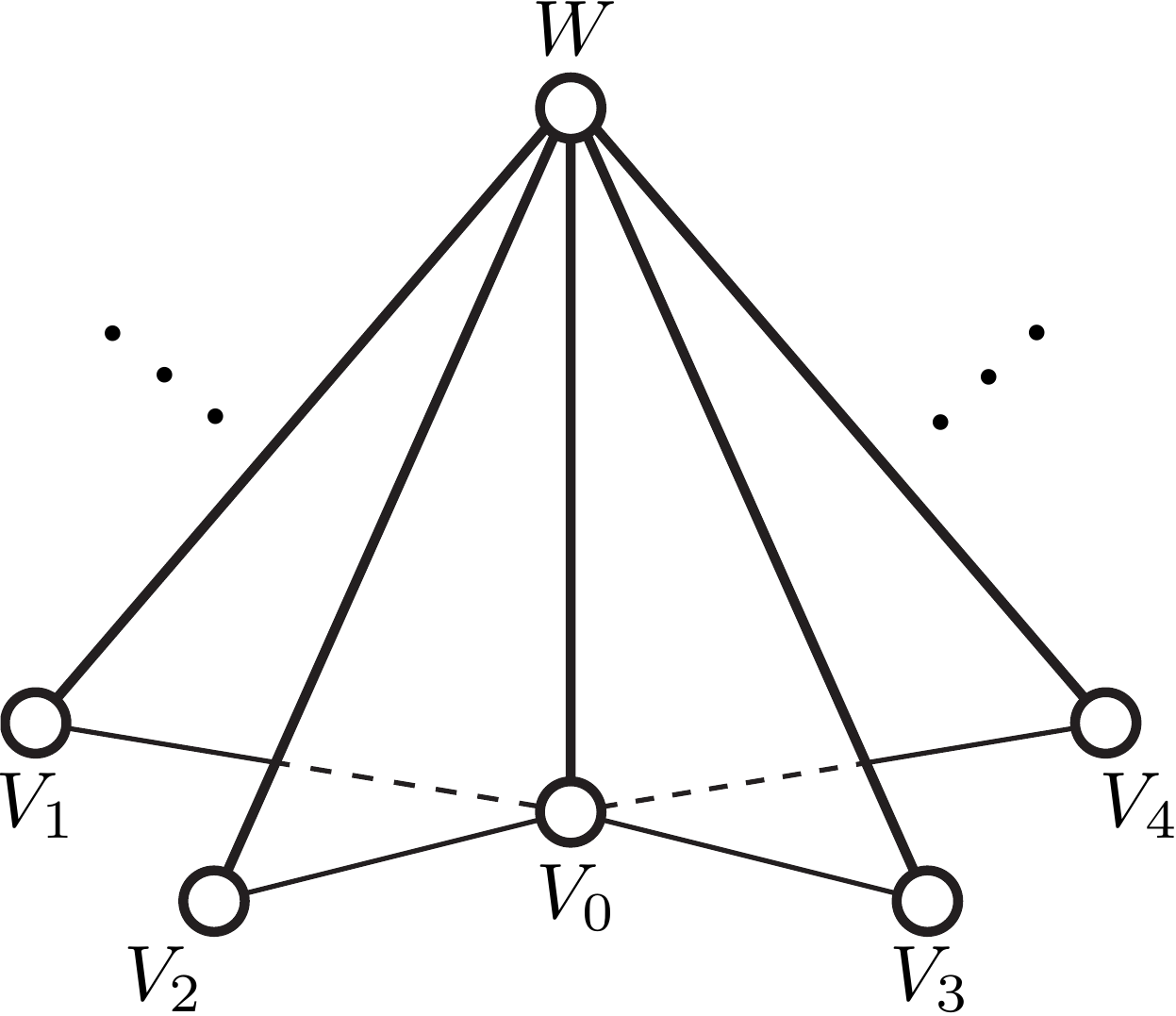}
\caption{An example of a $\V$-facial cluster in $\DVW(F)$. $(V_0, W)$ is the center and the other weak reducing pairs are hands. \label{figure3}}
\end{figure}

\begin{lemma}[J. Kim, Lemma 2.15 of \cite{10}]\label{lemma-2-15} 
Assume $M$ and $F$ as in Lemma \ref{lemma-2-13}.
Every $\V$-face belongs to some $\V$-facial cluster. 
Moreover, every $\V$-facial cluster has infinitely many hands.
\end{lemma}

The following lemma means that the generalized Heegaard splitting obtained by weak reduction along a weak reducing pair does not depend on the choice of the weak reducing pair if the weak reducing pair varies in a fixed $\V$- or $\W$-facial cluster.

\begin{lemma}[J. Kim, Lemma 2.17 of \cite{11}]\label{lemma-2-16}
Assume $M$ and $F$ as in Lemma \ref{lemma-2-13}.
Every weak reducing pair in a $\V$-face gives the same generalized Heegaard splitting after weak reduction up to isotopy.
Therefore, every weak reducing pair in a $\V$-facial cluster gives the same generalized Heegaard splitting after weak reduction up to isotopy.
Moreover, the embedding of the thick level contained in $\V$ or $\W$ does not vary in the relevant compression body up to isotopy. 
\end{lemma}

Here, we give an easy explanation of Lemma \ref{lemma-2-16}.
If we consider two $\V$-disks in the $\V$-face, then one is a meridian disk of the solid torus which the other cuts off from $\V$ by Lemma \ref{lemma-2-8} (see (a) of Figure \ref{figure4}).
We denote the former and the latter as $\bar{V}$ and $V$ respectively.
Since $V$ is a band-sum of two parallel copies of $\bar{V}$ in $\V$, it is easy to see that $F_{\bar{V}}$ is isotopic to the genus two component of $F_V$ in $\V$, say $\bar{F}_V$, when we push off them into the interior of $\V$ (see (b) of Figure \ref{figure4}).
\begin{figure}
\includegraphics[width=12cm]{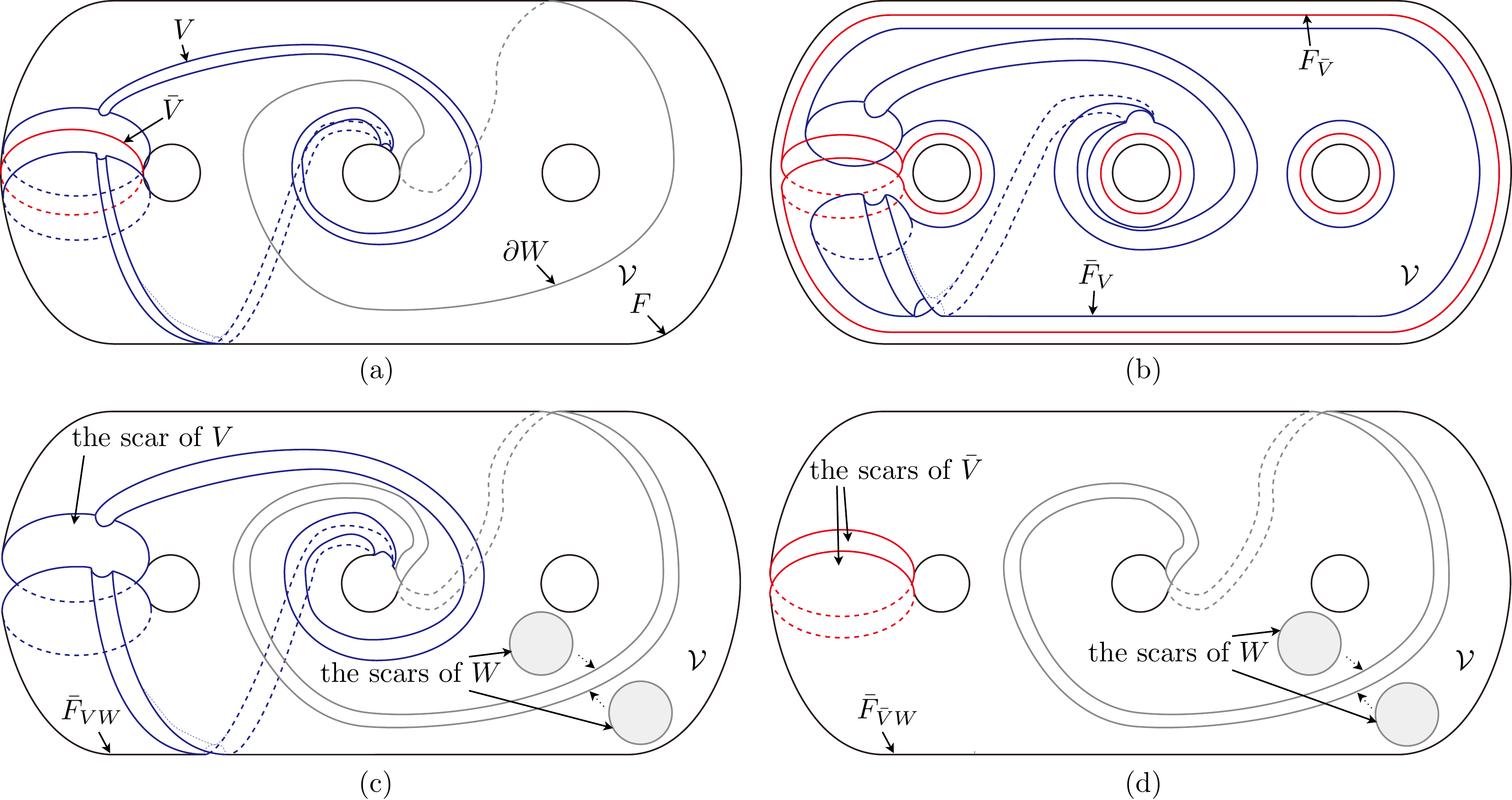}
\caption{$F_{\bar{V}}$ is isotopic to the genus two component of $F_{V}$ in the interior of $\V$. \label{figure4}}
\end{figure}
Moreover, we will see that the thick level contained in $\V$ of the generalized Heegaard splitting obtained by weak reduction from $(\V,\W;F)$ along a weak reducing pair would come from the the genus two component of the one obtained by compressing $F$ along the $\V$-disk of the weak reducing pair in Lemma \ref{lemma-3-1}.
Hence, this describes the third statement of Lemma \ref{lemma-2-16}.
In addition, if we consider that the inner thin level comes from the component of $F_{VW}$ having scars of both $V$ and $W$ for a weak reducing pair $(V,W)$ from Section \ref{section3}, say $\bar{F}_{VW}$, then we can see that the inner thin levels $\bar{F}_{VW}$ and $\bar{F}_{\bar{V}W}$ corresponding to two weak reducing pairs of the $\V$-face are also isotopic in $M$.

The next lemma gives an upper bound for the dimension of $\DVW(F)$ and restricts the shape of a $3$-simplex in $\DVW(F)$.

\begin{lemma}[J. Kim, Proposition 2.10 of \cite{9}]\label{lemma-2-17}
Assume $M$ and $F$ as in Lemma \ref{lemma-2-13}.
Then $\operatorname{dim}(\DVW(F))\leq 3$.
Moreover, if $\operatorname{dim}(\DVW(F))=3$, then every $3$-simplex in $\DVW(F)$ must have the form $\{V_1, V_2, W_1, W_2\}$, where $V_1, V_2\subset \V$ and $W_1,W_2\subset \W$.
Indeed, $V_1$ ($W_1$ resp.) is non-separating in $\V$ (in $\W$ resp.) and $V_2$ ($W_2$ resp.) is a band sum of two parallel copies of $V_1$ in $\V$ ($W_1$ in $\W$ resp.).
\end{lemma}

Note that the third statement of Lemma \ref{lemma-2-17} is obtained by applying Lemma \ref{lemma-2-8} to the $\V$-face $\{V_1,V_2,W_1\}$ and the $\W$-face $\{V_2,W_1,W_2\}$.

\section{The proof of Theorem \ref{theorem-1-2}\label{section3}}
In this section, we will prove Theorem \ref{theorem-1-2}.

The next lemma characterizes the possible generalized Heegaard splittings obtained by weak reductions from $(\V,\W;F)$ into five types. 
The proof is a routine exercise applying Definition \ref{definition-2-11} to all possible cases by considering Lemma \ref{lemma-2-7}.
That is, if there is a weak reducing pair $(V,W)$, then $\partial W$ is an essential simple closed curve in the genus two component of $F_V$ as in (a) of Figure \ref{figure4} and vise versa by Lemma \ref{lemma-2-7}.
Hence, we can easily generalize the way how $\partial V$ and $\partial W$ locate in $F$ and therefore we can do the pre-weak reduction along $(V,W)$ and cleaning for each case of Lemma \ref{lemma-3-1}.
We can use the diagrams in Figure 5 of \cite{11} or Figure \ref{figure6} to imagine possible weak reductions.

\begin{lemma}\label{lemma-3-1}
Suppose that $M$ is an irreducible $3$-manifold and $(\V,\W;F)$ is a weakly reducible, unstabilized, genus three   Heegaard splitting of $M$.
Let $(\V_1,\V_2;T_1)\cup(\W_1,\W_2;T_2)$ be the generalized Heegaard splitting obtained by weak reduction along a weak reducing pair $(V,W)$, where $\partial_-\V_2\cap \partial_-\W_1\neq\emptyset$.
Then this generalized Heegaard splitting is one of the following five types (see Figure \ref{figure5}).
\begin{enumerate}[(a)]
\item Each of $\partial_-\V_2$ and $\partial_-\W_1$ consists of a torus, where either\label{GHS-a}
	\begin{enumerate}[(i)]
	\item both $V$ and $W$ are non-separating in $\V$ and $\W$ respectively and $\partial V\cup\partial W$ is also non-separating in $F$,
	\item $V$ cuts off a solid torus from $\V$ and $W$ is non-separating in $\W$,
	\item $W$ cuts off a solid torus from $\W$ and $V$ is non-separating in $\V$, or
	\item each of $V$ and $W$ cuts off a solid torus from $\V$ or $\W$.
	\end{enumerate}
	We call it a ``\textit{type (a) GHS}''.
\item One of $\partial_-\V_2$ and $\partial_-\W_1$ consists of a torus and the other consists of two tori, where either\label{lemma-3-1-b}
	\begin{enumerate}[(i)]
	\item $V$ cuts off $(\text{torus})\times I$ from $\V$ and $W$ is non-separating in $\W$,\label{lemma-3-1-b-i}
	\item $V$ cuts off $(\text{torus})\times I$ from $\V$ and $W$ cuts off a solid torus from $\W$,\label{lemma-3-1-b-ii}
	\item $W$ cuts off $(\text{torus})\times I$ from $\W$ and $V$ is non-separating in $\V$, or\label{lemma-3-1-b-iii}
	\item $W$ cuts off $(\text{torus})\times I$ from $\W$ and $V$ cuts off a solid torus from $\V$.\label{lemma-3-1-b-iv}
	\end{enumerate}	
	We call it a ``\textit{type (b)-$\W$ GHS}'' for (\ref{lemma-3-1-b-i}) and (\ref{lemma-3-1-b-ii}) and ``\textit{type (b)-$\V$ GHS}'' for (\ref{lemma-3-1-b-iii}) and (\ref{lemma-3-1-b-iv}).
\item Each of $\partial_-\V_2$ and $\partial_-\W_1$ consists of two tori but $\partial_-\V_2\cap \partial_-\W_1$ is a torus, where each of $V$ and $W$ cuts off $(\text{torus})\times I$ from $\V$ or $\W$.\label{lemma-3-1-c}
We call it a ``\textit{type (c) GHS}''.
\item Each of $\partial_-\V_2$ and $\partial_-\W_1$ consists of two tori and $\partial_-\V_2\cap \partial_-\W_1$ also consists of two tori, where both $V$ and $W$ are non-separating in $\V$ and $\W$ respectively but $\partial V\cup\partial W$ is separating in $F$.\label{lemma-3-1-d}
We call it a ``\textit{type (d) GHS}''.
\end{enumerate}
\end{lemma}

\begin{figure}
\includegraphics[width=12cm]{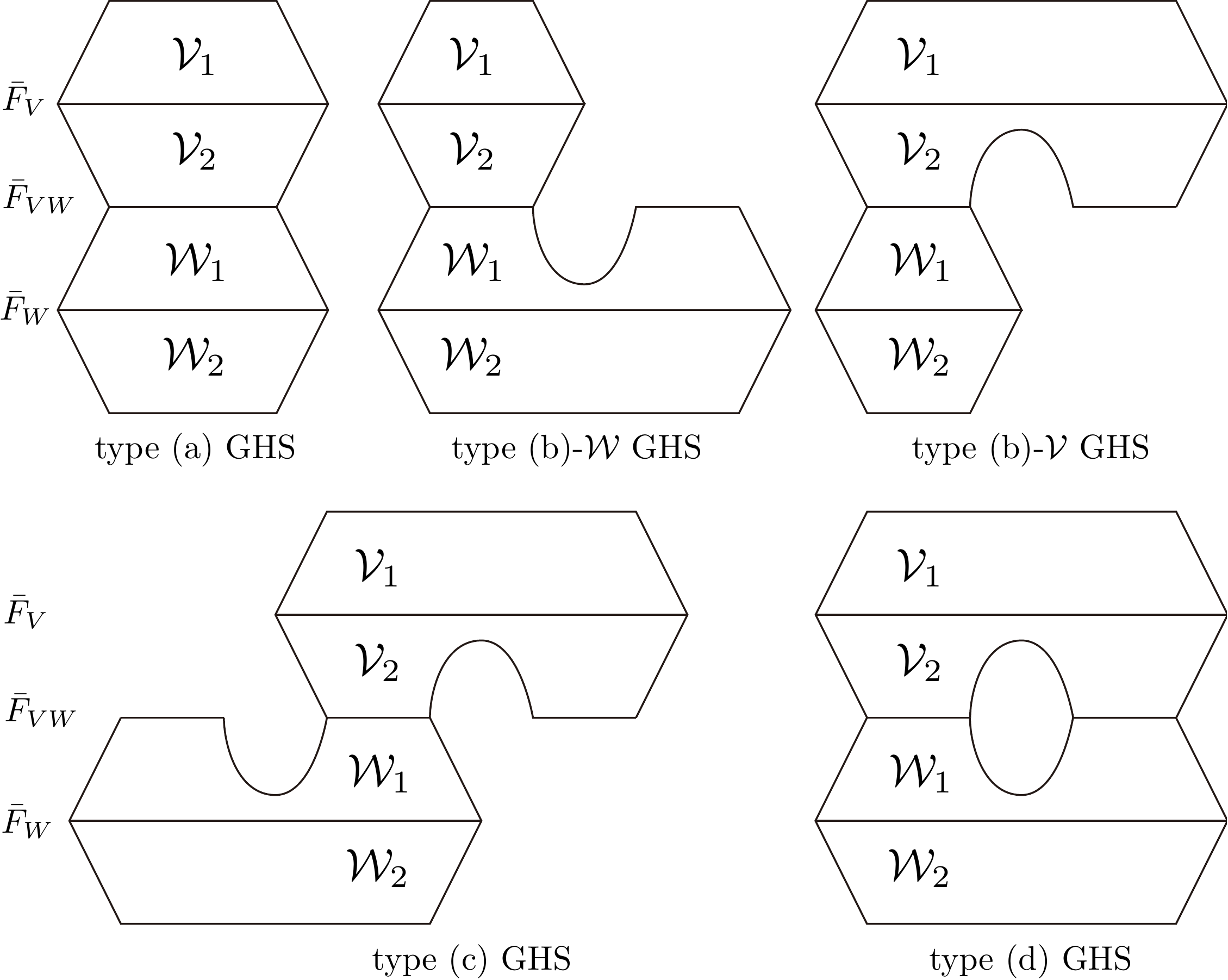}
\caption{the five types of generalized Heegaard splittings \label{figure5}}
\end{figure}

In Figure \ref{figure5}, the thick level $\bar{F}_V$ ($\bar{F}_W$ resp.) comes from the genus two component of $F_V$ ($F_W$ resp.) and the inner thin level $\bar{F}_{VW}$ comes from the component of $F_{VW}$ having scars of both $V$ and $W$ or the union of such components.
Here, $\bar{F}_{VW}$ consists of (i) a torus for the cases (\ref{GHS-a}), (\ref{lemma-3-1-b}) and (\ref{lemma-3-1-c}) of Lemma \ref{lemma-3-1} or (ii) two tori for the case (\ref{lemma-3-1-d}) of Lemma \ref{lemma-3-1}.
If we observe Figure \ref{figure6} describing the case (\ref{lemma-3-1-b-ii}), then we can see the reason why the inner thin level comes from the component of $F_{VW}$ having scars of both $V$ and $W$ or the union of such components.
By considering Definition \ref{definition-2-11}, we call $\bar{F}_V$ and $\bar{F}_W$ ``\textit{the thick level contained in $\V$}'' and  ``\textit{the thick level contained in $\W$}'' respectively as we previously used in the statement of Lemma \ref{lemma-2-16}.

\begin{figure}
\includegraphics[width=12cm]{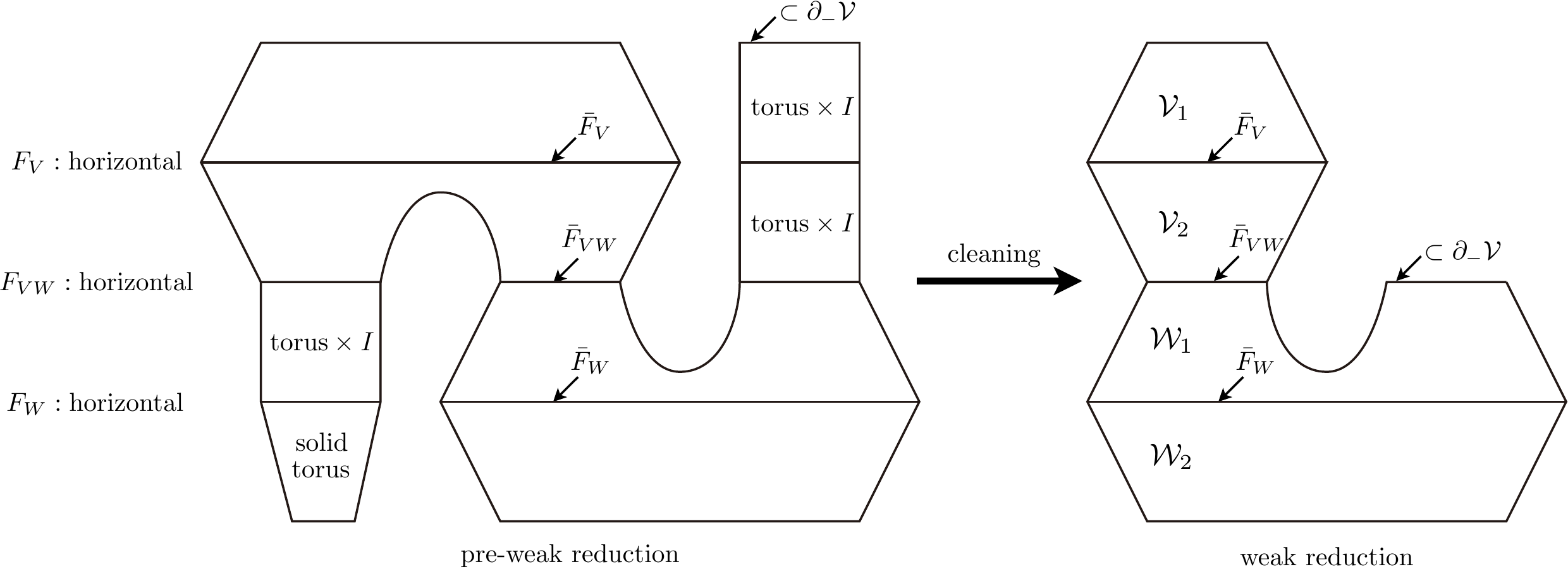}
\caption{The inner thin level comes from the component of $F_{VW}$ having scars of both $V$ and $W$. \label{figure6}}
\end{figure}

Since every weak reducing pair in a $\V$- or $\W$-facial cluster $\varepsilon$ gives a unique generalized Heegaard splitting after weak reduction up to isotopy by Lemma \ref{lemma-2-16}, we can say $\varepsilon$ has a GHS of either type (a), type (b)-$\W$ or type (b)-$\V$  by Lemma \ref{lemma-3-1} (we exclude the possibility that  $\varepsilon$ has a GHS of type (c) or type (d) by Lemma \ref{lemma-3-7}).

\begin{lemma}\label{lemma-3-2}
Assume $M$ and $F$ as in Lemma \ref{lemma-3-1}.
Let $\varepsilon_\V$ and $\varepsilon_\W$ be a $\V$-facial cluster and a $\W$-facial cluster such that they share the common center.
Then $\varepsilon_\V$ and $\varepsilon_\W$ have the same GHS (up to isotopy) of type (a).
\end{lemma}

\begin{proof}
Let $(\bar{V},\bar{W})$ be the common center.
Since $\varepsilon_\V$ and $\varepsilon_\W$ share the common center, if we apply Lemma \ref{lemma-2-16}, then it is obvious that $\varepsilon_\V$ and $\varepsilon_\W$ have the same GHS $\mathbf{H}$ up to isotopy.
Moreover, the common center $(\bar{V},\bar{W})$ consists of non-separating disks by Definition \ref{definition-2-14}.
Hence, a hand $(V,\bar{W})$ of $\varepsilon_\V$ consists of a separating disk $V$ and the non-separating disk $\bar{W}$.
But the assumption that the weak reducing pairs $(V,\bar{W})$ and $(\bar{V},\bar{W})$ give the same GHS up to isotopy forces $\mathbf{H}$ to be of type (a) by Lemma \ref{lemma-3-1}.

This completes the proof.
\end{proof}

\begin{definition}\label{definition-3-3}
Assume $M$ and $F$ as in Lemma \ref{lemma-3-1} and let $(V,W)$ be a weak reducing pair.
Suppose that the generalized Heegaard splitting obtained by weak reduction along $(V,W)$ is a type (a) GHS.
If $V$ is separating, then $V$ cuts off a solid torus from $\V$ by Lemma \ref{lemma-3-1} and we can choose a meridian disk $\bar{V}$ of the solid torus such that it misses $V$ and $W$ by Lemma \ref{lemma-2-7}.
Similarly, we can choose such non-separating disk $\bar{W}$ from $\W$ if $W$ is separating.
If both $V$ and $W$ are separating, then we can see that the once-punctured torus which $\partial V$ cuts off from $F$ misses the once-punctured torus which $\partial W$ cuts off from $F$ by Lemma \ref{lemma-2-7}.
This means that the solid torus which $V$ cuts off from $\V$ misses the solid torus which $W$ cuts off from $\W$, i.e. the four disks $V$, $\bar{V}$, $\bar{W}$ and $W$ are pairwise disjoint.
Choose $\bar{V}$ ($\bar{W}$ resp.) as $V$ ($W$ resp.) itself if $V$ ($W$ resp.) is non-separating.

Hence, we can find a weak reducing pair $(\bar{V},\bar{W})$ such that either
\begin{enumerate}
\item $(V,W)$ is $(\bar{V},\bar{W})$ itself,\label{definition-3-3-1}
\item $(V,W)$ is contained in a $\V$-face $\Delta_\V=\{V,\bar{V},\bar{W}\}$,\label{definition-3-3-2}
\item $(V,W)$ is contained in a $\W$-face $\Delta_\W=\{\bar{V},\bar{W},W\}$, or\label{definition-3-3-3}
\item $(V,W)$ is contained in a $3$-simplex $\Sigma_{VW}=\{V,\bar{V},\bar{W},W\}$.\label{definition-3-3-4}
\end{enumerate}
Here, $(\bar{V},\bar{W})$ is uniquely determined after once $(V,W)$ was given since the meridian disk of a solid torus is unique up to isotopy.

We will prove that $(V,W)$ is contained in a $3$-simplex of the form $$\Sigma_{V'W'}=\{V',\bar{V},\bar{W},W'\}$$ for $V'\subset \V$ and $W'\subset \W$ in any case as well as the case (\ref{definition-3-3-4}).
If we consider case (\ref{definition-3-3-2}), then (i) $\partial V$ divides $F$ into a once-punctured genus two surface $F'$ and a once-punctured torus $F''$ and (ii) $\partial \bar{V}\subset F''$ and $\partial\bar{W}\subset F'$ by Lemma \ref{lemma-2-8}.
Hence, we can find a band-sum $W'$ of two parallel copies of $\bar{W}$ in $\W$ such that $\partial W'\subset F'$, i.e. $(V,W)$ belongs to the $3$-simplex $\{V=V',\bar{V},W=\bar{W},W'\}$.
The symmetric argument also holds for the case (\ref{definition-3-3-3}).
If we consider the case (\ref{definition-3-3-1}), then it is easy to find two separating disks $V'$ and $W'$ such that $V'$ and $W'$ are band-sums of two parallel copies of $V$ and $W$ in $\V$ and $\W$ respectively and $\{V',V=\bar{V},W=\bar{W},W'\}$ forms a $3$-simplex in $\DVW(F)$ since $\partial V\cup\partial W$ is non-separating in $F$ by Lemma \ref{lemma-3-1}.

Hence, we get a $\V$-face $\{V',\bar{V},\bar{W}\}$ and a $\W$-face $\{\bar{V},\bar{W},W'\}$ from the $3$-simplex $\Sigma_{V'W'}$.
If we use Lemma \ref{lemma-2-15}, then we can guarantee the existence of the $\V$-facial cluster $\varepsilon_\V$ and the $\W$-facial cluster $\varepsilon_\W$ containing $(\bar{V},\bar{W})$ and both vertices $V$ and $W$ belong to $\varepsilon_\V\cup\varepsilon_\W$.

We can observe the follows.
\begin{enumerate}[(a)]
\item Only $\bar{V}$ and $\bar{W}$ are non-separating disks and the others are separating disks among the vertices of $\varepsilon_\V\cup\varepsilon_\W$ by Definition \ref{definition-2-14}.\label{definition-3-3-preclaim-a}
\item Since $\varepsilon_\V$ and $\varepsilon_\W$ share the common center $(\bar{V},\bar{W})$, each of $\bar{V}$ and $\bar{W}$ is connected to every the other vertex in $\varepsilon_\V\cup\varepsilon_\W$ by a $1$-simplex in $\varepsilon_\V\cup\varepsilon_\W$.\label{definition-3-3-preclaim-b}
\end{enumerate}
These two properties are at the base of the next claim.\\ 

\Claim{Let $\Sigma$ be the union of all simplices of $\DVW(F)$ spanned by the vertices of $\varepsilon_\V\cup\varepsilon_\W$.
Then $\Sigma$ is equal to $\bigcup_{V',W'}\Sigma_{V'W'}$ for all possible $V'$ and $W'$.}

\begin{proofc}
Since the four vertices of $\Sigma_{V'W'}$ come from a $\V$-face $\{V',\bar{V},\bar{W}\}$ in $\varepsilon_\V$ and a $\W$-face $\{\bar{V},\bar{W},W'\}$ in $\varepsilon_\W$, $\Sigma_{V'W'}$ is a $3$-simplex spanned by four vertices of $\varepsilon_\V\cup\varepsilon_\W$.
Therefore, $\bigcup_{V',W'}\Sigma_{V'W'}\subset \Sigma$ is obvious.
Hence, we will prove that $\Sigma\subset\bigcup_{V',W'}\Sigma_{V'W'}$.

It is sufficient to show that each simplex in $\Sigma$ belongs to some $\Sigma_{V'W'}$.
If there is a $0$-simplex in $\Sigma$, then it is contained in a $1$-simplex containing $\bar{V}$ or $\bar{W}$ in $\varepsilon_\V\cup\varepsilon_\W$ by (\ref{definition-3-3-preclaim-b}).
Hence, we will consider $1$- or more simplices in $\Sigma$.

\Case{a} $\sigma\subset\Sigma$ is a $1$-simplex.

If $\sigma$ contains a non-separating disk, then it must be $\bar{V}$ or $\bar{W}$ by (\ref{definition-3-3-preclaim-a}).
Since the other disk of $\sigma$ is also a vertex of $\varepsilon_\V\cup\varepsilon_\W$, $\sigma$ belongs to a $\V$-face of $\varepsilon_\V$ or a $\W$-face of $\varepsilon_\W$  by (\ref{definition-3-3-preclaim-b}).
Therefore, if we use the similar argument from the observation corresponding to the case (\ref{definition-3-3-2}) or (\ref{definition-3-3-3}) right before Claim, then we can find a $3$-simplex $\Sigma_{V'W'}$ in $\DVW(F)$ containing the $\V$- or $\W$-face and therefore $\Sigma_{V'W'}$ contains $\sigma$, leading to the result.

Hence, assume that $\sigma$ consists of two separating disks.
If $\sigma$ is not a weak reducing pair, i.e. it belongs to $\DV(F)$ or $\DW(F)$, then $\sigma$ is not contained in a $\V$-face or a $\W$-face by Lemma \ref{lemma-2-8} (in a $\V$-face, one of the $\V$-disks is non-separating and the other is separating).
Hence, we can find two $\V$-faces in $\varepsilon_\V$ or two $\W$-faces in $\varepsilon_\W$ such that each face contains one vertex of $\sigma$ by the assumption that the vertices of $\sigma$ come from those of $\varepsilon_\V\cup\varepsilon_\W$, i.e. $\sigma$ gives a $3$-simplex together with these two $\V$-faces or $\W$-faces.
But this $3$-simplex contains three vertices from $\DV(F)$ or $\DW(F)$, violating Lemma \ref{lemma-2-17}.
Hence, $\sigma$ is a weak reducing pair, i.e. it gives a $3$-simplex together with the $\V$-face $\Delta_\V$ of $\varepsilon_\V$ and the $\W$-face $\Delta_\W$ of $\varepsilon_\W$ such that each face contains one vertex of $\sigma$ and therefore this $3$-simplex has the form $\Sigma_{V'W'}$, leading to the result.

\Case{b} $\Delta\subset\Sigma$ is a $2$-simplex.

If $\Delta$ is not a $\V$- or $\W$-face, then it would consist of only three disks from $\DV(F)$ or $\DW(F)$.
Without loss of generality, assume that $\Delta\subset\DV(F)$.
Then each of three vertices of $\Delta$ is connected to $\bar{W}$ by a $1$-simplex by (\ref{definition-3-3-preclaim-b}).
That is, we get a $3$-simplex in $\DVW(F)$ containing three vertices coming from $\DV(F)$, violating Lemma \ref{lemma-2-17}.
Hence, $\Delta$ is a $\V$- or $\W$-face.
If $\Delta$ belongs to $\varepsilon_\V$ or $\varepsilon_\W$, then we can find a $3$-simplex $\Sigma_{V'W'}$ in $\DVW(F)$ containing $\Delta$ by using the similar 
argument from the observation corresponding to the case (\ref{definition-3-3-2}) or (\ref{definition-3-3-3}) right before Claim, leading to the result.
If $\Delta$ does not belong to $\varepsilon_\V\cup\varepsilon_\W$, then exactly one of $\bar{V}$ and $\bar{W}$ does not belong to $\Delta$ (Lemma \ref{lemma-2-8} forces $\Delta$ to have at least one non-separating disk, i.e. one of $\bar{V}$ and $\bar{W}$ by (\ref{definition-3-3-preclaim-a})).
Without loss of generality, assume that $\bar{V}\notin\Delta$ and $\bar{W}\in\Delta$.
In this case, $\Delta$ must be a $\W$-face otherwise it must contain $\bar{V}$ by (\ref{definition-3-3-preclaim-a}) and Lemma \ref{lemma-2-8}.
Hence, $\Delta$ gives a $3$-simplex of the form $\Sigma_{V'W'}$ together with $\bar{V}$ by (\ref{definition-3-3-preclaim-b}) and therefore $\Delta\subset \Sigma_{V'W'}$, leading to the result.

\Case{c} $\Sigma'\subset\Sigma$ is a $3$-simplex.

Then $\Sigma'$ must contain two non-separating disks (separating disks resp.) such that one comes from $\V$ and the other comes from $\W$ by Lemma \ref{lemma-2-17}, but the choice of non-separating disks among the vertices of $\varepsilon_\V\cup\varepsilon_\W$ is uniquely determined as $\{\bar{V},\bar{W}\}$ by (\ref{definition-3-3-preclaim-a}).
This means that $\Sigma'$ is of the form $\Sigma_{V'W'}$, leading to the result.

By Lemma \ref{lemma-2-17}, we do not need to consider more high dimensional simplices.
This completes the proof of Claim.
\end{proofc}

For a weak reducing pair $(V,W)$, we denote the generalized Heegaard splitting obtained by weak reduction from $(\V,\W;F)$ along $(V,W)$ as $\mathbf{H}_{(V,W)}$.
If a weak reducing pair $\sigma$ in $\Sigma$ belongs to $\varepsilon_\V$ or $\varepsilon_\W$, then $\mathbf{H}_\sigma$ is isotopic to $\mathbf{H}_{(\bar{V},\bar{W})}$ by Lemma \ref{lemma-2-16}.
Suppose that $\sigma$ does not belong to $\varepsilon_\V\cup\varepsilon_\W$.
By the proof of Claim, $\sigma$ belongs to a $3$-simplex of the form $\Sigma_{V'W'}=\{V',\bar{V},\bar{W},W'\}$ and therefore it would consist of two separating disks $V'$ and $W'$.
Consider the $\V$-face $\{V',\bar{V},W'\}$ containing $\sigma$ in $\Sigma_{V'W'}$.
Then $\mathbf{H}_{(V',W')}$ is isotopic to $\mathbf{H}_{(\bar{V},W')}$ by Lemma \ref{lemma-2-16}.
Moreover, $\mathbf{H}_{(\bar{V},W')}$ is isotopic to $\mathbf{H}_{(\bar{V},\bar{W})}$ by Lemma \ref{lemma-2-16} because $(\bar{V},W')$ belongs to the $\W$-face $\{\bar{V},\bar{W},W'\}\subset \varepsilon_\W$, i.e. $\mathbf{H}_{(V',W')}$ is isotopic to $\mathbf{H}_{(\bar{V},\bar{W})}$.
Since every $\Sigma_{V'W'}$ in $\bigcup_{V',W'}\Sigma_{V'W'}$ has the common weak reducing pair $(\bar{V},\bar{W})$,  the generalized Heegaard splitting obtained by weak reduction along a weak reducing pair in $\Sigma$ is unique up to isotopy.

We call $\Sigma$ ``\textit{a building block of $\DVW(F)$ having a type (a) GHS}''.
We call the unique weak reducing pair $(\bar{V},\bar{W})$ consisting of non-separating disks in $\Sigma$ the \textit{center} of the building block.

Note that Claim only needs a $\V$-facial cluster and a $\W$-facial cluster sharing the common center even though we started from a weak reducing pair giving a GHS of type (a) after weak reduction (the phrase ``\textit{the similar argument from the observation corresponding to the case (\ref{definition-3-3-2}) or (\ref{definition-3-3-3}) right before Claim}'' only needs a $\V$- or $\W$-face containing two non-separating disks without a specific kind of GHS).

In summary, we give the formal definition (the term ``type (a)'' comes from Lemma \ref{lemma-3-2}).

\textbf{The Formal Definition:}
A \textit{building block of $\DVW(F)$ having a type (a) GHS} is the union of all simplices of $\DVW(F)$ spanned by the vertices of $\varepsilon_\V\cup\varepsilon_\W$, where $\varepsilon_\V$ is a $\V$-facial cluster, $\varepsilon_\W$ is a $\W$-facial cluster, and they share the common center.
\end{definition}

\begin{figure}
\includegraphics[width=10cm]{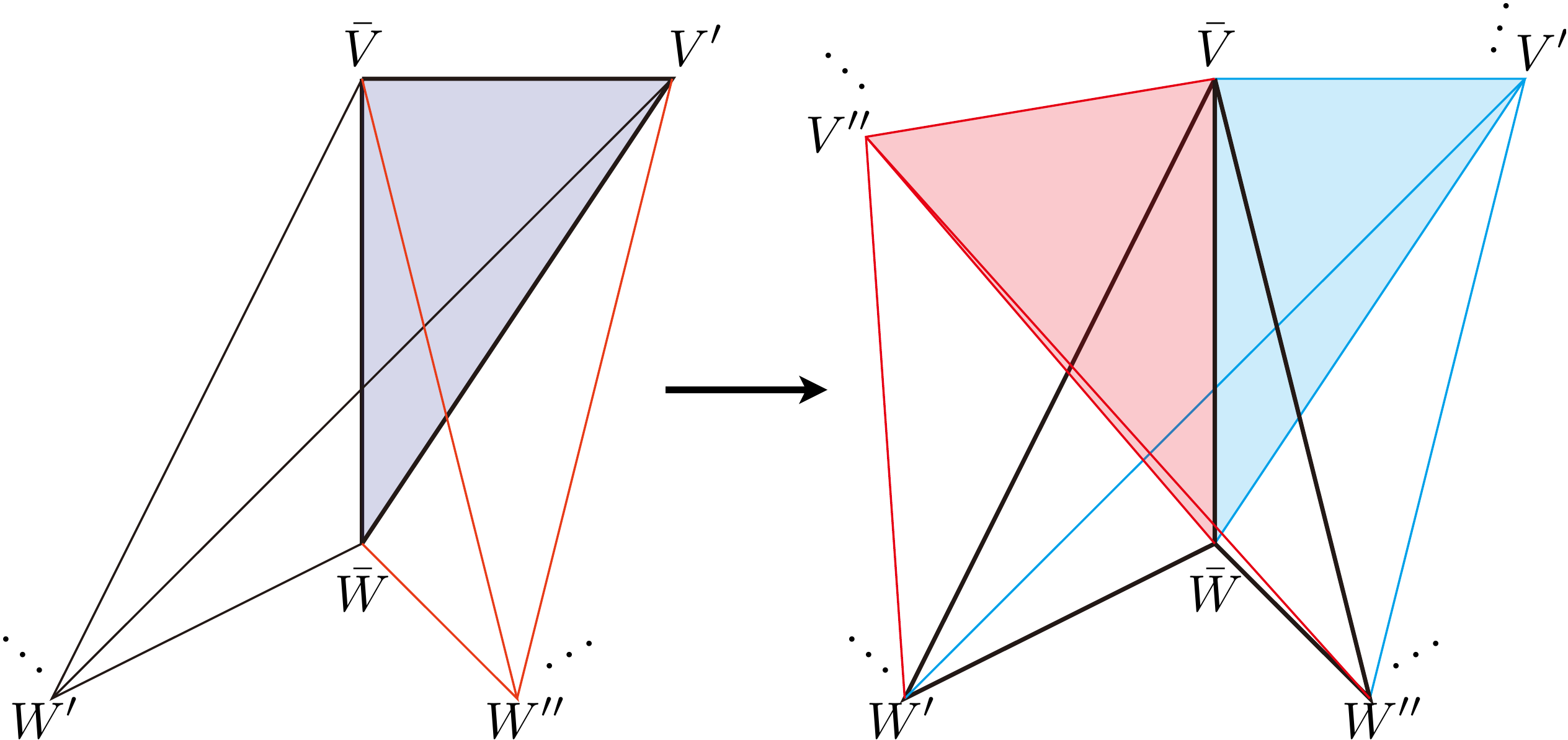}
\caption{Constucting $\Sigma$ \label{figure7}}
\end{figure}

Figure \ref{figure7} describes the way how we can imagine a building block having a type (a) GHS.
Since the building block $\Sigma$ is equal to $\bigcup_{V',W'}\Sigma_{V'W'}$ by Claim in Definition \ref{definition-3-3}, we consider the union $\bigcup_{W'}\Sigma_{V'W'}$ for each  $V'$ first  and then the whole union $\bigcup_{V'} \left(\bigcup_{W'}\Sigma_{V'W'}\right)$. 
Note that if there are two $3$-simplexes $\Sigma_{V'W'}$ and $\Sigma_{V''W''}$ such that $\Sigma_{V'W'}\neq \Sigma_{V''W''}$, then $\Sigma_{V'W'} \cap \Sigma_{V''W''}$ contains at least $\{ \bar{V},\bar{W} \}$. 
Hence, if the intersection is of two dimension, then a $2$-simplex in the intersection is a $\V$-face in $\varepsilon_\V$ or a $\W$-face in $\varepsilon_\W$.

If we consider $\DVW(F)$ of Case (a) in the proof of Theorem 1.1 of \cite{11}, then it is a building block of $\DVW(F)$ having a type (a) GHS itself.
Hence, if we use the same argument in \cite{11}, then we get the following lemma.

\begin{lemma}
A building block of $\DVW(F)$ having a type (a) GHS is contractible.
\end{lemma}

\begin{definition}\label{definition-3-5}
Assume $M$ and $F$ as in Lemma \ref{lemma-3-1} and let $(V,W)$ be a weak reducing pair.
Suppose that the generalized Heegaard splitting obtained by weak reduction along $(V,W)$ is a type (b)-$\W$ GHS.
Then $V$ cuts off $(\text{torus})\times I$ from $\V$ by Lemma \ref{lemma-3-1}.
Let $\bar{V}$ be $V$ itself.
Choose $\bar{W}$ as $W$ itself if $W$ is non-separating in $\W$.
If $W$ is separating, then it cuts off a solid torus from $\W$ by Lemma \ref{lemma-3-1}.
Here, we can choose a meridian disk $\bar{W}$ of the solid torus such that it misses $W$ and $V=\bar{V}$ by Lemma \ref{lemma-2-7}.
Hence, we can find a weak reducing pair $(\bar{V},\bar{W})$ such that either
\begin{enumerate}
\item $(V,W)$ is $(\bar{V},\bar{W})$ itself, or\label{definition-3-5-1}
\item $(V,W)$ is contained in a $\W$-face $\Delta_\W=\{\bar{V},\bar{W},W\}$.\label{definition-3-5-2}
\end{enumerate}
Here, $(\bar{V},\bar{W})$ is uniquely determined  after once $(V,W)$ was given since the meridian disk of a solid torus is unique up to isotopy.

Consider the case (\ref{definition-3-5-1}).
We can see that (i) $\partial \bar{V}$ divides $F$ into a once-punctured genus two surface $F'$ and a once-punctured torus $F''$ and (ii) $\partial\bar{W}\subset F'$ by Lemma \ref{lemma-2-7}.
Hence, it is easy to find a separating disk $W'$ such that $W'$ is a band-sum of two parallel copies of $W=\bar{W}$ and $\{V=\bar{V},W=\bar{W},W'\}$ forms a $\W$-face in $\DVW(F)$.
Hence, we can guarantee the existence of the $\W$-facial cluster $\varepsilon_\W$ containing $(V,W)$ whose center is  $(\bar{V}=V,\bar{W})$ in the case (\ref{definition-3-5-1}) as well as the case (\ref{definition-3-5-2}) by Lemma \ref{lemma-2-15}.
We call $\varepsilon_\W$ ``\textit{a building block of $\DVW(F)$ having a type (b)-$\W$ GHS}''.

Symmetrically, if the generalized Heegaard splitting obtained by weak reduction along $(V,W)$ is a type (b)-$\V$ GHS, then $W$ cuts off $(\text{torus})\times I$ from $\W$ by Lemma \ref{lemma-3-1} and we can find a $\V$-facial cluster $\varepsilon_\V$ containing $(V,W)$ and the uniquely determined weak reducing pair $(\bar{V},\bar{W}=W)$.
We call $\varepsilon_\V$ ``\textit{a building block of $\DVW(F)$ having a type (b)-$\V$ GHS}''.

We call the center $(\bar{V},\bar{W})$ of the $\W$- or $\V$-facial cluster corresponding to a building block of $\DVW(F)$ having a type (b)-$\W$ or type (b)-$\V$ GHS the \textit{center} of the building block.
By Lemma \ref{lemma-3-1}, the center of a building block of $\DVW(F)$ having a type (b)-$\W$ GHS (type (b)-$\V$ GHS resp.) consists of a separating $\V$-disk ($\W$-disk resp.) which cuts off $(\text{torus})\times I$ from $\V$ ($\W$ resp.) and a non-separating $\W$-disk ($\V$-disk resp.).
Since a $\W$- or $\V$-facial cluster is contractible, a building block of $\DVW(F)$ having a type (b)-$\W$ or type (b)-$\V$ GHS is contractible.

In summary, we give the formal definition.

\textbf{The Formal Definition:} 
\begin{enumerate}
\item \textit{A building block of $\DVW(F)$ having a type (b)-$\W$ GHS} is a $\W$-facial cluster having a type (b)-$\W$ GHS.
\item \textit{A building block of $\DVW(F)$ having a type (b)-$\V$ GHS} is a $\V$-facial cluster having a type (b)-$\V$ GHS.
\end{enumerate}
\end{definition}

\begin{definition}\label{definition-3-6}
Assume $M$ and $F$ as in Lemma \ref{lemma-3-1} and let $(V,W)$ be a weak reducing pair.
Suppose that the generalized Heegaard splitting obtained by weak reduction along $(V,W)$ is a type (c) GHS (type (d) GHS resp.).
In this case, we call the weak reducing pair $(V,W)$ itself ``\textit{a building block of $\DVW(F)$ having a type (c) GHS} (\textit{type (d) GHS} resp.)''.
We also define the \textit{center} of the building block $(\bar{V},\bar{W})$ as $(V,W)$ itself.
\end{definition}

In brief, we can find a building block of $\DVW(F)$ containing an arbitrarily given weak reducing pair.

Now we introduce the next lemma describing a special property for a building block of $\DVW(F)$ having a type (c) or type (d) GHS.

\begin{lemma}\label{lemma-3-7}
Assume $M$ and $F$ as in Lemma \ref{lemma-3-1}.
Then a building block of $\DVW(F)$ having a type (c) or type (d) GHS cannot be contained in a $\V$- or $\W$-face.
\end{lemma}

\begin{proof}
Suppose that a building block of $\DVW(F)$ having a type (c) GHS is contained in a $\V$-face $\Delta$ without loss of generality.
This building block is just a weak reducing pair $(V,W)$ consisting of two separating disks by Definition \ref{definition-3-6}.
Hence, the third vertex $V'$ of $\Delta$ other than $(V,W)$ is a non-separating disk in $\V$ by Lemma \ref{lemma-2-8}.
But the weak reducing pair $(V',W)$ cannot give a type (c) GHS after weak reduction  by Lemma \ref{lemma-3-1} since it consists of a non-separating disk and a separating disk, i.e. the two generalized Heegaard splittings obtained by weak reductions from $(\V,\W;F)$ along $(V,W)$ and $(V',W)$ respectively are non-isotopic. 
This violates Lemma \ref{lemma-2-16}.

We can use the symmetric argument for a building block of $\DVW(F)$ having a type (d) GHS.
This completes the proof.
\end{proof}

By definition, every weak reducing pair belongs to some building block of $\DVW(F)$.
But it is not clear that $\DVW(F)$ is the union of all building blocks since there is a possibility that some simplex of $\DVW(F)$ might not belong to a building block even though every weak reducing pair of this simplex belongs to some building block.
Hence, we need the following lemma.

\begin{lemma}\label{lemma-3-8}
Every simplex in $\DVW(F)$ belongs to a building block of $\DVW(F)$.
\end{lemma}

\begin{proof}
Every $0$-simplex in $\DVW(F)$ is a vertex of $1$- or more dimensional simplex of $\DVW(F)$ by definition of $\DVW(F)$.
Hence, we will consider $1$- or more dimensional simplices of $\DVW(F)$.

If there is a $1$-simplex in $\DVW(F)$, then either it is a weak reducing pair or a subsimplex of a $2$- or more dimensional simplex in $\DVW(F)$.
But we already knew that every weak reducing pair belongs to a building block of $\DVW(F)$ by the definitions of building blocks of $\DVW(F)$.
Hence, we only need to consider $2$- or $3$-simplices in $\DVW(F)$ by Lemma \ref{lemma-2-17}.

\Case{a} 
Suppose that $\Delta\subset\DVW(F)$ is a $2$-simplex.
If $\Delta$ is neither a $\V$-face nor a $\W$-face, then it belongs to $\DV(F)$ or $\DW(F)$ and it is a subsimplex of a $3$-simplex in $\DVW(F)$ by Lemma \ref{lemma-2-17}.
But this contradicts Lemma \ref{lemma-2-17} since this $3$-simplex contains three vertices from $\DV(F)$ or $\DW(F)$.

Without loss of generality, assume that $\Delta$ is a $\V$-face.
Then there is a $\V$-facial cluster $\varepsilon$ containing $\Delta$ by Lemma \ref{lemma-2-15}.

If the $\W$-disk of the center of $\varepsilon$ is non-separating, then $\Delta$ consists of two weak reducing pairs, where one consists of non-separating disks and the other consists of a non-separating $\W$-disk and a separating $\V$-disk by Definition \ref{definition-2-14}.
Hence, the generalized Heegaard splitting corresponding to $\Delta$ must be of type (a) by Lemma \ref{lemma-3-1}.
That is, we can consider $\Delta$ as $\Delta_\V$ in the case (\ref{definition-3-3-2}) of Definition \ref{definition-3-3} and therefore we can find a building block of $\DVW(F)$ having a type (a) GHS as in Definition \ref{definition-3-3} and it contains $\Delta$, leading to the result.

Hence, we can assume that the $\W$-disk of the center of $\varepsilon$ is separating, i.e. $\Delta$ consists of two weak reducing pairs, where one consists of a non-separating $\V$-disk and a separating $\W$-disk and the other consists of separating disks.
This means that the generalized Heegaard splitting corresponding to $\Delta$ is a type (a) GHS or a type (b)-$\V$ GHS by Lemma \ref{lemma-3-1}.
In the latter case, $\varepsilon$ itself is the building block of $\DVW(F)$ having a type (b)-$\V$ GHS as in Definition \ref{definition-3-5}, leading to the result.
In the former case, the $\W$-disk cuts off a solid torus from $\W$ by Lemma \ref{lemma-3-1} and we can find a meridian disk of the solid torus missing the three disks from $\Delta$ by Lemma \ref{lemma-2-7} and Lemma \ref{lemma-2-8}.
This gives a $3$-simplex $\Sigma_{V'W'}=\{V',\bar{V},\bar{W},W'\}$ containing $\Delta$, where $\bar{V}$ and $\bar{W}$ are non-separating disks and $V'$ and $W'$ are separating disks.
Here, we can see $\Delta=\{V',\bar{V},W'\}$.
Let $\Delta_\V$ and $\Delta_\W$ be the $\V$-face $\{V',\bar{V},\bar{W}\}$ and the $\W$-face $\{\bar{V},\bar{W},W'\}$.
Let $\varepsilon_\V$ and $\varepsilon_\W$ be the $\V$- and $\W$-facial clusters containing $\Delta_\V$ and $\Delta_\W$ respectively guaranteed by Lemma \ref{lemma-2-15}.
Since $\varepsilon_\V$ and $\varepsilon_\W$ share the common center $(\bar{V},\bar{W}$), we can define $\Sigma$ as the set of all simplices of $\DVW(F)$ spanned by the vertices of $\varepsilon_\V\cup\varepsilon_\W$ and it would be the building block of $\DVW(F)$ having a type (a) GHS containing $\Delta$ ($\Delta\subset \Sigma_{V'W'}\subset \bigcup_{V',W'}\Sigma_{V'W'}=\Sigma$), leading to the result.

\Case{b} $\Sigma'\subset \DVW(F)$ is a $3$-simplex.

If we consider a $3$-simplex $\Sigma'$, then it must be of the form $\Sigma'=\{V',\bar{V},\bar{W},W'\}$ by Lemma \ref{lemma-2-17}, where $\bar{V}$ ($\bar{W}$ resp.) is a non-separating disk in $\V$ ($\W$ resp.) and $V'$ ($W'$ resp.) is a band-sum of two parallel copies of $\bar{V}$ in $\V$ ($\bar{W}$ in $\W$ resp.).
Hence, we can find a $\V$-facial cluster and a $\W$-facial cluster sharing the common center $(\bar{V},\bar{W})$ as the previous paragraph, and these give a  building block $\Sigma$ of $\DVW(F)$ having a type (a) GHS containing $\Sigma'$ ($\Sigma'\subset \bigcup_{V',W'}\Sigma_{V'W'}=\Sigma$), leading to the result.

This completes the proof.
\end{proof}

Note that the intersection of a building block of $\DVW(F)$ having a type (a) GHS and $\DV(F)$ is the same as that of the $\V$-facial cluster containing the center of the building block and $\DV(F)$ (see Claim of Definition \ref{definition-3-3} and Figure \ref{figure7}).
Hence, if a building block of $\DVW(F)$ intersects $\DV(F)$ in a set of dimension one, then it is a $\ast$-shaped graph with infinitely many edges (see Lemma \ref{lemma-2-15}).
That is, every $1$-simplex in the intersection contains the $\V$-disk of the center of the building block  and this vertex is positioned at the center of the $\ast$-shaped graph.

In summary, we get the following lemma.

\begin{lemma}\label{lemma-3-9}
Assume $M$ and $F$ as in Lemma \ref{lemma-3-1}.
Then $\DVW(F)$ is the union of building blocks of $\DVW(F)$.
Moreover, every building block of $\DVW(F)$ is contractible.
When a building block of $\DVW(F)$ intersects $\DV(F)$ or $\DW(F)$, the dimension of the intersection is at most one.
Indeed, the intersection is either a vertex if the dimension is zero or a $\ast$-shaped graph with infinitely many edges if the dimension is one.
Therefore, if the dimension of the intersection is one, then the $\ast$-shaped graph comes from the intersection of the $\V$- or $\W$-facial cluster in the building block containing the center and $\DV(F)$ or $\DW(F)$, i.e. the vertex positioned at the center of the $\ast$-shaped graph comes from the center of the building block.
\end{lemma}

Now we introduce the next lemma.

\begin{lemma}\label{lemma-3-10}
Assume $M$ and $F$ as in Lemma \ref{lemma-3-1}.
If two building blocks of $\DVW(F)$ have GHSs of different types, then they cannot intersect each other.
\end{lemma}

\begin{proof}
Suppose that there are two building blocks $\mathcal{B}_1$ and $\mathcal{B}_2$ such that they have GHSs of different types and $\mathcal{B}_1\cap\mathcal{B}_2\neq\emptyset$.

Suppose that both $\mathcal{B}_1$ and $\mathcal{B}_2$ are of dimension one, i.e. one has a type (c) GHS and the other has a type (d) GHS.
Since a building block of $\DVW(F)$ having a type (c) GHS consists of separating disks, it cannot intersect a building block of $\DVW(F)$ having a type (d) GHS consisting of non-separating disks.
Hence, we can assume that at least one between $\mathcal{B}_1$ and $\mathcal{B}_2$, say $\mathcal{B}_1$, is of dimension at least two, i.e. $\mathcal{B}_1$ has a type (a), type (b)-$\W$ or type (b)-$\V$ GHS.
This means that there is a $\V$- or $\W$-facial cluster contained in $\mathcal{B}_1$.\\

\Claim{There exist two different weak reducing pairs such that one belongs to $\mathcal{B}_1$,  the other belongs to $\mathcal{B}_2$ and one shares a vertex with the other.}

\begin{proofc}
For any vertex in a building block of $\DVW(F)$, there is a weak reducing pair in the building block containing the vertex by definition.
Hence, if $\dim(\mathcal{B}_1\cap\mathcal{B}_2)=0$, then the proof is obvious.
From now on, assume that $\dim(\mathcal{B}_1\cap\mathcal{B}_2)\geq 1$.
If $\mathcal{B}_1\cap\mathcal{B}_2$ contains a weak reducing pair $\delta$, then there is a $\V$- or $\W$-face $\Delta$ in $\mathcal{B}_1$ containing $\delta$ by the definition of a building block having a type (a), type (b)-$\W$ or type (b)-$\V$ GHS (consider $\delta \subset \Sigma_{V'W'}$ in the proof of Claim in Definition \ref{definition-3-3} for type (a) GHS).
Therefore, we can find two different weak reducing pairs $\delta$ and $\delta'$ in $\Delta$ sharing a vertex such that $\delta\subset \mathcal{B}_1\cap\mathcal{B}_2$ and $\delta'\subset\mathcal{B}_1$, leading to the result.

Hence, assume that every $1$-simplex of $\mathcal{B}_1\cap\mathcal{B}_2$ is not a weak reducing pair.
Without loss of generality, assume that there is a $1$-simplex $\delta$ in $\mathcal{B}_1\cap\mathcal{B}_2$ belonging to $\DV(F)$.
Since $\mathcal{B}_1$ and $\mathcal{B}_2$ have GHSs of different types and the dimension of the intersection of them in $\DV(F)$ is at least one, we can assume that $\mathcal{B}_1$ has a type (a) GHS and $\mathcal{B}_2$ has a type (b)-$\V$ GHS without loss of generality.
Moreover, Lemma \ref{lemma-3-9} guarantees that  we can choose a $\V$-face $\Delta_1\subset \mathcal{B}_1$ containing the center $(\bar{V}_1,\bar{W}_1)$ of $\mathcal{B}_1$ and a $\V$-face $\Delta_2\subset \mathcal{B}_2$ containing the center $(\bar{V}_2,\bar{W}_2)$ of $\mathcal{B}_2$ such that $\delta\subset\Delta_1$ and $\delta\subset\Delta_2$.
Hence, $\delta$ contains $\bar{V}_1$ and $\bar{V}_2$.
But if $\bar{V}_1\neq\bar{V}_2$, then $\delta$ consists of two non-separating disks, violating Lemma \ref{lemma-2-8} by considering two $\V$-disks of $\Delta_1$.
Therefore, we get $\bar{V}_1=\bar{V}_2$, i.e. the centers of $\mathcal{B}_1$ and $\mathcal{B}_2$ share a vertex $\bar{V}_1=\bar{V}_2$. 
Since $\mathcal{B}_1$ has a type (a) GHS and $\mathcal{B}_2$ has a type (b)-$\V$ GHS, $\bar{W}_1$ is non-separating and $\bar{W}_2$ is separating in $\W$, i.e. two centers are different weak reducing pairs.
This completes the proof of Claim.
\end{proofc}

By Claim and Lemma 8.4 of \cite{2}, there is a sequence $\Delta_0$, $\cdots$, $\Delta_n$ of $\V$- and $\W$-faces such that $\Delta_0$ contains a weak reducing pair of $\mathcal{B}_1$, $\Delta_n$ contains a weak reducing pair of  $\mathcal{B}_2$, and $\Delta_{i-1}$ and $\Delta_i$ share a weak reducing pair for $1\leq i \leq n$.
But if we use Lemma \ref{lemma-2-16} inductively from $\Delta_0$ to $\Delta_n$, then $\mathcal{B}_1$ and $\mathcal{B}_2$ have the same GHS up to isotopy.
This means that $\mathcal{B}_1$ and $\mathcal{B}_2$ have GHSs of the same type, violating the assumption. 
This completes the proof.
\end{proof} 

\begin{lemma}\label{lemma-3-11}
Assume $M$ and $F$ as in Lemma \ref{lemma-3-1}.
If two different building blocks of $\DVW(F)$ have GHSs of the same type, then the centers of them cannot intersect each other.

\end{lemma}

\begin{proof}
Let $\mathcal{B}_1$ and $\mathcal{B}_2$ be different building blocks having GHSs of the the same type.
Suppose that the centers of them intersect each other.

\Case{a} $\mathcal{B}_1$ and $\mathcal{B}_2$ have type (a) GHSs.

If the centers of them are the same, then the $\V$-facial cluster in  $\mathcal{B}_1$ containing the center is the same as the $\V$-facial cluster in  $\mathcal{B}_2$ containing the center.
The previous argument also holds for the $\W$-facial clusters in $\mathcal{B}_1$ and $\mathcal{B}_2$ containing the common center.
But this implies that two building blocks are the same by Definition \ref{definition-3-3}, violating the assumption.

Hence, one center shares only one vertex with the other.
Let $(\bar{V}_1,\bar{W}_1)$ and $(\bar{V}_2,\bar{W}_2)$ be the centers of $\mathcal{B}_1$ and $\mathcal{B}_2$ respectively and assume that $\bar{V}_1=\bar{V}_2$ and $\bar{W}_1$ is not isotopic to $\bar{W}_2$ in $\W$ without loss of generality.
Here, Lemma 8.4 of \cite{2} implies that there exists a sequence of $\V$- and $\W$-faces $\Delta_0$, $\cdots$, $\Delta_n$ such that $(\bar{V}_1,\bar{W}_1)$ is contained in $\Delta_0$, $(\bar{V}_2,\bar{W}_2)$ is contained in $\Delta_n$, and $\Delta_{i-1}$ and $\Delta_i$ share a weak reducing pair for $1\leq i \leq n$.
Assume that $n$ is such a smallest integer.
Then there is a simplicial map $\iota: (\text{triangulated})~D^2 \to \DVW(F)$ such that $\iota(\partial D^2)$ contains $(\bar{V}_1,\bar{W}_1)\cup(\bar{V}_2,\bar{W}_2)$ and $\iota(D^2)=\Delta_0\cup\cdots\cup \Delta_n$, where each triangle of $D^2$ corresponds to exactly one of $\Delta_i$ for some $0\leq i \leq n$ and we denote this triangle as $\bar{\Delta}_i$, i.e. $\iota(\bar{\Delta}_i)=\Delta_i$ (see Figure \ref{figure8}, a white one is mapped into a $\V$-face and a shaded one is mapped into a $\W$-face).
\begin{figure}
\includegraphics[width=6cm]{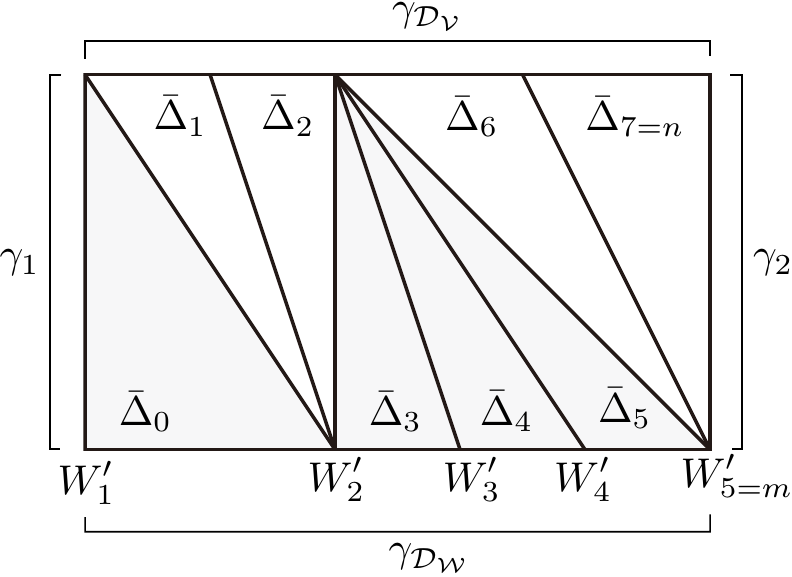}
\caption{a triangulation of $D^2$ \label{figure8}}
\end{figure}
Here, $\gamma_1=\bar{\Delta}_0\cap\partial D^2$ and $\gamma_2=\bar{\Delta}_n\cap\partial D^2$ are mapped into $(\bar{V}_1,\bar{W}_1)$ and $(\bar{V}_2,\bar{W}_2)$ by $\iota$ respectively.
We can see that $n\geq 1$ otherwise $\Delta_0$ is a $\W$-face containing two non-separating $\W$-disks $\bar{W}_1$ and $\bar{W}_2$, violating Lemma \ref{lemma-2-8}.
Hence, $\operatorname{cl}(\partial D^2-(\gamma_1\cup\gamma_2))$ consists of two parts (possibly empty) $\gamma_{\mathcal{D}_\V}$  and $\gamma_{\mathcal{D}_\W}$, where $\iota(\gamma_{\mathcal{D}_\V})\subset \DV(F)$ and $\iota(\gamma_{\mathcal{D}_\W})\subset \DW(F)$.
Note that $\gamma_{\mathcal{D}_\W}\neq \emptyset$ since $\bar{W}_1$ is not isotopic to $\bar{W}_2$.
Let $W_1'-\cdots-W_m'$ be $\gamma_{\mathcal{D}_\W}$, where each edge $W_{i-1}'-W_i'$ is mapped into a $1$-simplex in $\DW(F)$ by $\iota$, so $\iota(W_1')=\bar{W}_1$ and $\iota(W_m')=\bar{W}_2$.
We can see that $\iota(W_i')$ is non-separating for odd $i$ and it is separating for even $i$ since (i) $\iota(W_1')=\bar{W}_1$, i.e. a non-separating disk and (ii) $\iota(W_{i-1}'-W_i')$ is an edge of a $\W$-face consisting of $\W$-disks and therefore one of $\iota(W_{i-1}')$ and $\iota(W_i')$ is non-separating and the other is separating by Lemma \ref{lemma-2-8}.
Since $\iota(W_1')=\bar{W}_1$ is not isotopic to $\iota(W_m')=\bar{W}_2$, we get $m\geq 3$.
If we consider $\iota(W_3')$, then $\iota(W_2'-W_3')$ is an edge of a $\W$-face consisting of $\W$-disks and therefore $\iota(W_3')$ comes from the meridian disk of the solid torus which $\iota(W_2')$ cuts from $\W$ by Lemma \ref{lemma-2-8}.
This argument also holds for $\iota(W_1')$ by considering $\iota(W_1'-W_2')$.
But the meridian disk of a solid torus is uniquely determined up to isotopy, i.e. $\iota(W_3')$ is isotopic to $\iota(W_1')=\bar{W}_1$ in $\W$.
If we repeat this argument until we meet $\iota(W_m')=\bar{W}_2$, then we conclude that $\bar{W}_1$ is isotopic to $\bar{W}_2$, violating the assumption.

\Case{b} $\mathcal{B}_1$ and $\mathcal{B}_2$ have type (b)-$\W$ GHSs (symmetrically type (b)-$\V$ GHSs).

If we use the same argument in the first paragraph of Case a, then we can see that one center shares only one vertex with the other.
Recall that the the center of a building block of $\DVW(F)$ having a type (b)-$\W$ GHS consists of a separating $\V$-disk which cuts off $(\text{torus})\times I$ from $\V$ and a non-separating $\W$-disk.

If these two centers have the common $\V$-disk, then the other two disks of the two centers are non-separating $\W$-disks.
Hence, if we use the same arguments in the second paragraph of Case a, then the two centers are the same, violating the assumption.

Therefore, two centers have the common $\W$-disk.
If we use Lemma 8.4 of \cite{2}, there is a sequence $\Delta_0$, $\cdots$, $\Delta_n$ of $\V$- and $\W$-faces such that $\Delta_0$ contains the center of $\mathcal{B}_1$, $\Delta_n$ contains the center of $\mathcal{B}_2$, and $\Delta_{i-1}$ and $\Delta_i$ share a weak reducing pair for $1\leq i \leq n$.
Assume that $n$ is such a smallest integer.
If $n=0$, then $\Delta_0$ contains two separating disks from $\V$, violating Lemma \ref{lemma-2-8}.
Therefore, we get $n\geq 1$.

Suppose that $\Delta_0$ is a $\V$-face.
In this case, the third vertex of $\Delta_0$ other than the center of $\mathcal{B}_1$ must be a meridian disk of the solid torus which the separating $\V$-disk of $\Delta_0$ cuts off from $\V$ by Lemma \ref{lemma-2-8}, violating the assumption that it cuts off $(\text{torus})\times I$ from $\V$.

Hence, $\Delta_0$ is a $\W$-face.
Let us consider $\Delta_1$. 
Since $n$ is such a smallest integer, the third vertex of $\Delta_0$ other than the center of $\mathcal{B}_1$ must belong to $\Delta_1$.
Moreover, the third vertex is a separating $\W$-disk by Lemma \ref{lemma-2-8}.
Hence, the weak reducing pair $\Delta_0\cap\Delta_1$ consists of separating disks.
Assume that $\Delta_1$ is a $\V$-face.
Then the third vertex of $\Delta_1$ other than $\Delta_0\cap \Delta_1$ must be a meridian disk of the solid torus which the separating $\V$-disk of $\Delta_1$ cuts off from $\V$ by Lemma \ref{lemma-2-8}.
But the separating $\V$-disk of $\Delta_1$ comes from the $\V$-disk of the center of $\mathcal{B}_1$, violating the assumption that it cuts off $(\text{torus})\times I$ from $\V$.
Therefore, $\Delta_1$ is a $\W$-face.
From $\Delta_0\cup \Delta_1$, we get a sequence of $\W$-disks $W_0$, $W_1$, $W_2$, where $\{W_0,W_1\}\subset \Delta_0$, $\{W_1,W_2\}\subset \Delta_1$.
Here, the assumption that $n$ is such a smallest integer means that $W_0\neq W_2$.
Since two different $\W$-faces $\Delta_0$ and $\Delta_1$ share a weak reducing pair, the $\W$-disk of the common weak reducing pair $\Delta_0\cap\Delta_1$ is non-separating by Lemma \ref{lemma-2-13}, i.e. $W_1$ is non-separating, violating that $\Delta_0\cap\Delta_1$ consists of separating disks.

\Case{c} $\mathcal{B}_1$ and $\mathcal{B}_2$ have type (c) GHSs.

In this case, each of $\mathcal{B}_1$ and $\mathcal{B}_2$ is the center itself.
Hence, these two centers cannot be the same, i.e. two centers share a vertex.
If we use Lemma 8.4 of \cite{2}, then there is a sequence $\Delta_0$, $\cdots$, $\Delta_n$ of $\V$- and $\W$-faces such that $\Delta_0$ contains the center of $\mathcal{B}_1$, $\Delta_n$ contains the center of $\mathcal{B}_2$, and $\Delta_{i-1}$ and $\Delta_i$ share a weak reducing pair for $1\leq i \leq n$.
But this violates Lemma \ref{lemma-3-7}.

\Case{d} $\mathcal{B}_1$ and $\mathcal{B}_2$ have type (d) GHSs.

If we use the same arguments in Case c, then we get a contradiction.

This completes the proof.
\end{proof}

\begin{lemma}\label{lemma-3-12}
Assume $M$ and $F$ as in Lemma \ref{lemma-3-1}.
If two different building blocks of $\DVW(F)$ have GHSs of the same type, then they cannot intersect each other.
\end{lemma}

\begin{proof}
Let $\mathcal{B}_1$ and $\mathcal{B}_2$ be different building blocks having GHSs of the the same type and assume that $\mathcal{B}_1\cap\mathcal{B}_2\neq\emptyset$.
Since a building block of $\DVW(F)$ having a type (c) or type (d) GHS is the center of the building block itself by Definition \ref{definition-3-6}, they cannot intersect each other by Lemma \ref{lemma-3-11}.
Hence, we only need to consider when $\mathcal{B}_1$ and $\mathcal{B}_2$ have type (a), type (b)-$\V$ or type (b)-$\W$ GHSs. 

\Case{a} $\mathcal{B}_1$ and $\mathcal{B}_2$ have type (a) GHSs.

Without loss of generality, we can assume that there is a vertex $V\in \DV(F)$ in $\mathcal{B}_1\cap\mathcal{B}_2$.
If $V$ is a non-separating disk, then $V$ must be the $\V$-disk of the center of $\mathcal{B}_i$ for $i=1,2$ by Definition \ref{definition-3-3}, i.e. the centers of $\mathcal{B}_1$ and $\mathcal{B}_2$ intersect each other, violating Lemma \ref{lemma-3-11}.
Hence, $V$ is a separating disk, i.e. there are two $\V$-faces $\Delta_1\subset\mathcal{B}_1$ and $\Delta_2\subset\mathcal{B}_2$ such that $\Delta_i$ contains $V$ and the center of $\mathcal{B}_i$ for $i=1,2$ by Lemma \ref{lemma-3-9}.
Let $\bar{V}_i$ be the $\V$-disk of the center of $\mathcal{B}_i$ for $i=1,2$.
Then $\bar{V}_i$ is a meridian disk of the solid torus which $V$ cuts off from $\V$ for $i=1,2$ by applying Lemma \ref{lemma-2-8} to $\Delta_1$ and $\Delta_2$.
But the meridian disk in a solid torus is unique up to isotopy, i.e. $\bar{V}_1$ is isotopic to $\bar{V}_2$ in $\V$.
This means that the centers of $\mathcal{B}_1$ and $\mathcal{B}_2$ intersect each other, violating Lemma \ref{lemma-3-11}.

\Case{b} $\mathcal{B}_1$ and $\mathcal{B}_2$ have type (b)-$\V$ GHSs (symmetrically type (b)-$\W$ GHSs).

If there is a vertex $W\in \DW(F)$ in $\mathcal{B}_1\cap\mathcal{B}_2$, then the centers of $\mathcal{B}_1$ and $\mathcal{B}_2$ intersect each other since $\mathcal{B}_1$ and $\mathcal{B}_2$ are just $\V$-facial clusters, violating Lemma \ref{lemma-3-11}.
Hence, the intersection must come from $\DV(F)$ and choose a disk $V$ in the intersection.
If we use the same arguments as in Case a, then we get a contradiction.

This completes the proof.
\end{proof}

The next theorem describes each component of $\DVW(F)$ exactly.

\begin{theorem}\label{theorem-3-13}
Let $(\V,\W;F)$ be a weakly reducible, unstabilized, genus three Heegaard splitting in an orientable, irreducible $3$-manifold $M$.
Then every component of $\DVW(F)$ is just a building block of $\DVW(F)$.\end{theorem}

\begin{proof}
Let us consider a component $\mathcal{C}$ of $\DVW(F)$.
Then $\mathcal{C}$ is contained in a union of building blocks of $\DVW(F)$  by Lemma \ref{lemma-3-8}.
But Lemma \ref{lemma-3-10} and Lemma \ref{lemma-3-12} imply  that there cannot be two or more adjacent building blocks contained in $\mathcal{C}$.
This completes the proof.
\end{proof}

Theorem \ref{theorem-3-13} means that there is a function from the components of $\DVW(F)$ to the isotopy classes of the generalized Heegaard splittings obtained by weak reductions from $(\V,\W;F)$ since every weak reducing pair in a building block of $\DVW(F)$ gives the same generalized Heegaard splitting after weak reduction up to isotopy.

The next lemma determines all centers of building blocks of $\DVW(F)$.

\begin{lemma}\label{lemma-3-14}
Assume $M$ and $F$ as in Lemma \ref{lemma-3-1}.
A weak reducing pair $(V,W)$ of $(\V,\W;F)$ is the center of a building block of $\DVW(F)$ if and only if each of $V$ and $W$ does not cut off a solid torus from the relevant compression body.
Moveover, every component of $\DVW(F)$ can be represented by a uniquely determined weak reducing pair, i.e. the center of the corresponding building block of $\DVW(F)$.
\end{lemma}

\begin{proof}
By the definitions of building blocks of $\DVW(F)$ and Lemma \ref{lemma-3-1}, we can see that every disk in the center of a building block must not cut off a solid torus from the relevant compression body.
Moreover, if we consider the definitions of building blocks of $\DVW(F)$ and Lemma \ref{lemma-3-1} again, then any weak reducing pair which is not the center of a building block must contain a disk which is a band sum of two parallel copies of a non-separating disk in the center, i.e. this disk cuts off a solid torus from the relevant compression body.
This completes the proof of the former statement.
The latter statement is obvious by Theorem \ref{theorem-3-13}.
This completes the proof.
\end{proof}

The next lemma means that there might be more than one building block corresponding to an isotopy class of generalized Heegaard splittings obtained by weak reduction, where it is a reinterpretation of Lemma 3.4 of \cite{11} in the sense of Theorem \ref{theorem-3-13}.

\begin{lemma}[Lemma 3.4 of \cite{11}]\label{lemma-3-15}
Let $(\V,\W;F)$ be a weakly reducible, unstabilized, genus three Heegaard splitting in an orientable, irreducible $3$-manifold $M$.
Suppose that there are two generalized Heegaard splittings $\mathbf{H}_1$ and $\mathbf{H}_2$ obtained by weak reductions along $(V_1,W_1)$ and $(V_2,W_2)$ from $(\V,\W;F)$ respectively such that the thick levels of $\mathbf{H}_1$ and $\mathbf{H}_2$ contained in one compression body are non-isotopic in the compression body. (It may be possible that $\mathbf{H}_1$ is the same as $\mathbf{H}_2$ in $M$ up to isotopy.)
Then the building block of $\DVW(F)$ containing $(V_1,W_1)$ is different from that containing $(V_2,W_2)$.
\end{lemma}

\begin{proof}
Since (i) we already proved that the the generalized Heegaard splitting obtained by weak reduction is unique up to isotopy in a building block of $\DVW(F)$ by using Lemma \ref{lemma-2-16} in the definitions of building blocks of $\DVW(F)$ and (ii) Lemma \ref{lemma-2-16} also guarantees that the embeddings of the thick levels do not vary in the relevant compression bodies up to isotopy, the contrapositive holds obviously.

This completes the proof.
\end{proof}

Finally, we reach Theorem \ref{theorem-3-16}

\begin{theorem}[the Structure Theorem]\label{theorem-3-16}
Let $(\V,\W;F)$ be a weakly reducible, unstabilized, genus three Heegaard splitting in an orientable, irreducible $3$-manifold $M$.
Then there is a function from the components of $\DVW(F)$ to the isotopy classes of the generalized Heegaard splittings obtained by weak reductions from $(\V,\W;F)$.
The number of components of the preimage of an isotopy class of this function is the number of ways to embed the thick level contained in $\V$ into $\V$ (or in $\W$ into $\W$).
This means that if we consider a generalized Heegaard splitting $\mathbf{H}$ obtained by weak reduction from $(\V,\W;F)$, then the way to embed the thick level of $\mathbf{H}$ contained in $\V$ into $\V$ determines the way to embed the thick level of $\mathbf{H}$ contained in $\W$ into $\W$ up to isotopy and vise versa.
\end{theorem}

\begin{proof}
By the definitions of building blocks of $\DVW(F)$ and Theorem \ref{theorem-3-13}, the first statement is obvious.
Hence, we will prove the second statement.

Suppose that there are two generalized Heegaard splittings $\mathbf{H}_1$ and $\mathbf{H}_2$ obtained by weak reductions from $(\V,\W;F)$ along $(V_1,W_1)$ and $(V_2,W_2)$ respectively and the thick level of $\mathbf{H}_1$ contained in $\V$, say $T_\V^1$, is isotopic to that of $\mathbf{H}_2$ contained in $\V$, say $T_\V^2$, in $\V$.
Without changing the embeddings of the thick levels of $\mathbf{H}_1$ and $\mathbf{H}_2$ in the relevant compression bodies up to isotopy, we can assume that $(V_i,W_i)$ is the center of the corresponding building block of $\DVW(F)$, say $\mathcal{B}_i$, for $i=1,2$ by Lemma \ref{lemma-3-15}.

Suppose that $\mathcal{B}_1$ and $\mathcal{B}_2$ are different.
Now we isotope $T_\V^2$ into $T_\V^1$ in $\V$ to treat them as the same surface.
Let $\V'$ be the solid between $T_\V^1$ and $\partial_+\V$.
Then $\V'$ is a genus three compression body with at least one minus boundary component of genus two and $V_1$ and $V_2$ are compressing disks of $\V'$ by construction.
If $\partial_-\V'$ is connected, then we can see that every separating disk of $\V'$ must cut off a solid torus from $\V'$.
Hence, Lemma \ref{lemma-3-14} forces $V_1$ and $V_2$ to be non-separating disks of $\V'$.
But there is a unique non-separating disk in $\V'$ up to isotopy since $\V'$ is a genus three compression body with minus boundary consisting of a genus two surface.
Hence, $V_1$ is isotopic to $V_2$ in $\V'$ therefore so in $\V$.
But this means that the centers of $\mathcal{B}_1$ and $\mathcal{B}_2$ intersect each other, violating  Theorem \ref{theorem-3-13}.
If $\partial_-\V'$ is disconnected, then $V_1$ and $V_2$ must be isotopic to the unique compressing disk in $\V'$ since $\V'$ is a genus three compression body with minus boundary consisting of a genus two surface and a torus.
Hence, we also get a contradiction similarly.

Therefore, $\mathcal{B}_1$ and $\mathcal{B}_2$ are the same building block, i.e. the embedding of the thick level contained in $\W$ of $\mathbf{H}_1$ is also isotopic to that of $\mathbf{H}_2$ by Lemma \ref{lemma-3-15}.

Hence, we only need to count the number of ways to embed the thick level contained in $\V$ into $\V$ for an isotopy class of a generalized Heegaard splitting obtained by weak reduction from $(\V,\W;F)$.
Since Lemma \ref{lemma-3-15} means that two different embeddings correspond to  two different building blocks, the the number of ways to embed the thick level contained in $\V$ into $\V$ is exactly the same as the number of components of $\DVW(F)$ by Theorem \ref{theorem-3-13}.
This completes the proof of the second statement.

This completes the proof.
\end{proof}

\section*{Acknowledgments}
This research was supported by BK21 PLUS SNU Mathematical Sciences Division.


\begin{thebibliography}{0}

\bibitem{1} D. Bachman, Critical Heegaard surfaces, 
{\it Trans. Amer. Math. Soc}. {\bf 354} (2002), 4015--4042. 

\bibitem{2} D. Bachman, Connected sums of unstabilized Heegaard splittings are unstabilized, 
{\it Geom. Topol}. {\bf 12} (2008), 2327--2378.

\bibitem{3} D. Bachman, Topological index theory for surfaces in 3-manifolds, 
{\it Geom. Topol}.  {\bf 14} (2010), 585--609.

\bibitem{4} D. Bachman, Normalizing Topologically Minimal Surfaces I: Global to Local Index, arXiv:1210.4573.

\bibitem{5} D. Bachman, Normalizing Topologically Minimal Surfaces II: Disks,  arXiv:1210.4574.

\bibitem{6} D. Bachman, Normalizing Topologically Minimal Surfaces III: Bounded Combinatorics,  arXiv:1303.6643.

\bibitem{7} D. Bachman, Stabilizing and destabilizing Heegaard splittings of sufficiently complicated $3$-manifolds, 
{\it Math. Ann}. {\bf 355} (2013), 697--728.

\bibitem{8} J. Hempel, $3$-manifolds as viewed from the curve complex, \textit{Topology} \textbf{40} (2001), 631--657.

\bibitem{9} J. Kim, On critical Heegaard splittings of tunnel number two composite knot exteriors, {\it J. Knot Theory Ramifications} \textbf{22} (2013), 1350065, 11 pp.

\bibitem{10} J. Kim, On unstabilized genus three critical Heegaard surfaces, {\it Topology Appl}.  \textbf{165} (2014), 98--109.

\bibitem{11} J. Kim, A topologically minimal, weakly reducible, unstabilized Heegaard splitting of genus three is critical, arXiv:1402.4253.

\bibitem{12} M. Lustig and Y. Moriah, Closed incompressible surfaces in complements of wide knots and links, 
{\it Topology. Appl}. {\bf 92} (1999), 1--13.

\bibitem{13} D. McCullough, Virtually geometrically finite mapping class groups of $3$-manifolds, 
{\it J. Differential Geom}. {\bf 33} (1991) 1--65.

\bibitem{14} T. Saito, M. Scharlemann and J. Schultens, Lecture notes on generalized Heegaard splittings, arXiv:math/0504167v1.

\bibitem{15} M. Scharlemann and A. Thompson, Thin position for $3$-manifolds, \textit{AMS Contemp. Math.} {\bf 164} (1994), 231--238.

\end{thebibliography}
\end{document}